\documentclass[review,times]{elsarticle}

\usepackage{amssymb,amsmath,amsfonts,amscd,verbatim}
\usepackage{color,graphicx}
\usepackage{algorithm,algorithmic}
\usepackage{upgreek,comment,url}
\usepackage{mystylefile}
\usepackage{float}
\usepackage{rotating,fancyhdr}
\usepackage{booktabs}

\newtheorem{theorem}{Theorem}
\newtheorem{lemma}{Lemma}
\newproof{proof}{Proof}
\newcommand{\changes}[1]{\textcolor{black}{#1}}

\newcommand{\ii}[1]{\textcolor{red}{#1}}

\begin{document}
\begin{frontmatter}

\title{Many (most?) column subset selection criteria\\ are NP hard \changes{for a few columns}\tnoteref{t1}}
\tnotetext[t1]{The work was funded, in part, by the National Science Foundation through the awards DMS-1745654  and DMS-2411198, and the Department of Energy through the award DE-SC0023188.}

\author[1]{Ilse C.F. Ipsen}\ead{ipsen@ncsu.edu}
\affiliation{organization = {Department of Mathematics, North Carolina State University},
city = {Raleigh, NC}, country = {USA}}

\author[2]{Arvind K. Saibaba}\ead{asaibab@ncsu.edu}
\affiliation{organization = {Department of Mathematics, North Carolina State University},
city = {Raleigh, NC}, country = {USA}}


\begin{abstract}
 We consider a variety of criteria for selecting $k$ representative columns from a real $m\times n$ matrix $\ma$, \changes{when sufficiently few columns are required, i.e., $1\leq k\leq \min\{\rank(\ma), m/3\}$.} The criteria include the following optimization problems:
absolute volume and
S-optimality maximization; norm, pseudo-inverse norm,
and condition minimization number in the two-norm, Frobenius norm and Schatten $p$-norms for $p>2$; stable rank maximization; and the new
criterion of relative volume maximization, \changes{which is inversely proportional to a power of the condition number.}
We show that these criteria 
are NP hard and \changes{many} do not admit polynomial time approximation schemes (PTAS).
To formulate the optimization problems as decision problems, we derive optimal values for the subset selection criteria, as well as expressions for partitioned pseudo-inverses. \changes{The results for minimization of the pseudo-inverse in the Frobenius norm are applicable to trace optimization in A-optimal design.}
\end{abstract}

\newpageafter{abstract}
\end{frontmatter}

\section{Introduction}\label{s_intro}
Given a general\footnote{We make no assumptions on the dimensions $m$ and $n$ of $\ma$.}
matrix $\ma\in\rmn$, what does it mean to find $k$
representative columns  $\mc\in\real^{m\times k}$? This problem is known as
`column subset selection'.
Our goal is to establish the complexity classes for the popular subset selection problems in Table~\ref{t_intro},
\changes{in the case when sufficiently few columns are required, i.e., $1\leq k\leq \min\{\rank(\ma),m/3\}$.}

We show that most criteria in Table~\ref{t_intro} are NP hard, which means they cannot be solved in polynomial time
unless $P=NP$. 
Although most subset selection criteria are NP hard,
not all are. For example, a submatrix $\mc$ with minimal Frobenius norm can be computed
in polynomial time\footnote{A greedy algorithm picks the columns of $\ma$ with smallest two-norm.} (Section~\ref{s_twop}).

Our strategy is the following: Instead of considering the optimization problem, e.g., find a submatrix $\mc\in\real^{m\times k}$ with maximal volume,
we show NP hardness instead for the associated decision problem, e.g., given $b>0$, find a submatrix $\mc\in\real^{m\times k}$ whose
volume exceeds $b$. 
To connect the optimization and decision problems,
we derive optimal values for the subset selection
criteria (Section~\ref{s_unitnorm}), so that the optimization problems can be formulated
as decision problems for the optimal values.

\begin{comment}
According to \cite[Section 8.1]{KT2006}
\begin{quote}
\textit{From the point of view of polynomial-time solvability, there is not a significant difference between the optimization version of the problem [...] and the decision version [...]. Given a method to solve the optimization version, we automatically solve the
decision version (for any $b$) as well.}
\end{quote}
\end{comment}

Having established that a criterion is NP hard, it is reasonable to ask if it can be \textit{approximated} in polynomial time. This leads to the concept of a \textit{Polynomial Time 
Approximation Scheme (PTAS)}.
%We show that almost all criteria are NP hard, and do not 
%admit a \textit{Polynomial Time 
%Approximation Scheme (PTAS)}.
A PTAS is a polynomial
time algorithm such that for each $\epsilon>0$
 there exists a $(1+\epsilon)$-approximation algorithm for minimization problems, and a $(1-\epsilon)$-approximation algorithm for maximization problems 
 \cite[Section 1]{Papa1991}, \cite[Definition 1.2]{WillBook}. We show that most criteria 
 in Table~\ref{t_intro} do not admit a PTAS either.

 \subsection{Complexity Theory}
 We focus on decision problems with a binary output `yes' or `no', and adopt the Turing model of computation,
%the Cook-Karp-Levin model of complexity involving NP-Hardness and NP-Complete problems,
where inputs are rational numbers, and the input size is measured in terms of bits.  \changes{The only irrational number used in our reduction 
is $1/\sqrt{3}$, which can also be expressed in terms of a finite number of bit \cite[Section 9]{HLim13}.}

The complexity class $P$ consists of decision problems that can be solved by a deterministic Turing machine in polynomial time, while the complexity class $NP$ consists of decision problems for which a deterministic Turing machine can verify an output of `yes' 
or `no' in polynomial time.  A problem is \textit{NP hard} if there is polynomial time reduction to this problem from any problem in NP.
A decision problem is \textit{NP complete} if it is in NP and it is also NP hard.  
Thus, an NP hard problem is at least as hard as an NP complete problem \cite[Section 5.1]{Garey}.
%Not all NP hard problems are decision problems.

A standard strategy for establishing NP hardness performs a polynomial time reduction from an NP complete problem 
to the problem at hand. This reduction makes it possible to convert any algorithm  for solving the problem at hand into a corresponding algorithm for an NP complete problem.
In particular, a reduction that \textit{operates in polynomial time} makes it possible to convert any polynomial time 
for the problem at hand to a corresponding polynomial time algorithm for the NP complete problem 
\cite[Section 1.5]{Garey}.
Following \cite{CM09}, we choose
the NP complete problem `Exact cover by 3 sets (X3C)' \cite[Section A.3.1]{Garey},
which appears in Karp's 1972 list of NP complete problems \cite[Section 4]{Karp1972}.

To show that a subset selection criterion does not admit a PTAS, we create a ``gap'' between the optimal value and 
 a particular suboptimal value, the so-called
 \textit{inapproximability threshold}, and show that it is NP hard to approximate the criterion beyond the threshold, that is,
 for any value inside the gap.

 Unfortunately, creating such a gap is not always feasible~\cite[Section 1]{Papa1991}, and we were
 not able to do so for a few criteria, including
 Schatten 
 $p$-norm minimization for $p>2$ (Section~\ref{s_twop}) and 
 condition number minimization in Schatten $p$-norms
 for $p>2$ (Section~\ref{s_condno}). 

\begin{sidewaystable}\label{t_intro}
\centering
\fbox{%
\begin{tabular}{|l|l|l|l|c|}\hline
Quantity& Criterion & NP hard & No PTAS& Definition\\
\hline\hline
Volume & $\max_{\mc}{\vol(\mc)}$ & \cite[Theorem 4]{CM09} & \cite[Theorem 6]{CM09}$^{(1)}$
&$\vol(\mc)\equiv\prod_{j=1}^k{\sigma_j(\mc)}$\\
& & & \changes{\cite[Theorem 8]{DEFM15}} & \\
& & Theorem~\ref{t_voldec}& Theorem~\ref{t_vol} &\\
 & $\max_{\mc}{\rvol(\mc)}$ & 
Theorem~\ref{thm:rvolnp}& Theorem~\ref{t_volptas} &
 $\rvol(\mc)\equiv\vol(\mc)/\|\mc\|_2^k$\\
 &$\max_{\mc}{\sopt(\mc)}$ &Theorem~\ref{t_voldec} &
 Theorem~\ref{t_sopt} &$\sopt(\mc)\equiv(\vol(\mc)/\prod_{i=1}^k{\|\mc\ve_i\|_2)^{1/k}}$\\
Norm & $\min_{\mc}{\|\mc\|_2}$ & \cite[Theorem 4]{CM09} & \cite[Theorem~5]{CM09}& \\
&  & Theorem~\ref{t_norm} & Theorem~\ref{t_2m} &\\
 & $\min_{\mc}{\|\mc\|_F}$ & No & N/A&\\
 & $\min_{\mc}{\|\mc\|_{(p)}}$ & Theorem~\ref{t_norm}&&\\
Pseudo-inverse & $\min_{\mc}{\|\mc^{\dagger}\|_2}$ & \cite[Theorem 4]{CM09} & \cite[Theorem 7]{CM09} &\\
&& Theorem~\ref{t_pi7}& Theorem~\ref{t_ssv}&\\
 & $\min_{\mc}{\|\mc^{\dagger}\|_F}$ & Theorem~\ref{t_pi7} & Theorem~\ref{t_pi} &\\
 &$\min_{\mc}{\|\mc^{\dagger}\|_{(p)}}$& Theorem~\ref{t_pi7}&Theorem~\ref{t_pi3}&\\
Condition no& $\min_{\mc}{\kappa_2(\mc)}$ &\cite[Theorem 4]{CM09} & \cite[Theorem 8]{CM09}&
$\kappa_2(\mc)\equiv \|\mc\|_2\|\mc^{\dagger}\|_2$ \\
& & Theorem~\ref{t_cond}& Theorem~\ref{t_2cond}& \\
 & $\min_{\mc}{\kappa_F(\mc)}$ & 
Theorem~\ref{t_cond}& Theorem~\ref{t_fcond}&
$\kappa_F(\mc)\equiv \|\mc\|_F\|\mc^{\dagger}\|_F$\\
 & $\min_{\mc}{\kappa_D(\mc)}$ &
Theorem~\ref{t_cond} & Theorem~\ref{t_cd} & $\kappa_D(\mc)\equiv\|\mc\|_F\|\mc^{\dagger}\|_2$ \\ 
& $\min_{\mc}{\kappa_{D,p}(\mc)}$ & Theorem~\ref{t_cond} &&
$\kappa_{D,p}(\mc)\equiv \|\mc\|_{(p)}\|\mc^{\dagger}\|_2$\\
Stable rank & $\max_{\mc}{\sr(\mc)}$ & 
Theorem~\ref{t_other} &Theorem~\ref{t_sr}& $\sr(\mc)\equiv \|\mc\|_F^2/\|\mc\|_2^2$\\
&$\max_{\mc}{\sr_{(p)}(\mc)}$ & Theorem~\ref{t_other}&&
$\sr_{(p)}(\mc)\equiv \|\mc\|_{(p)}^p/\|\mc\|_2^p$
\\
Residual & $\min_{\mc}{\rho_2(\mc)}$ & &  &
$\rho_2(\mc)\equiv \|(\mi-\mc\mc^{\dagger})\ma\|_2$\\
 & $\min_{\mc}{\rho_F(\mc)}$ & 
\cite[Theorem 2.2]{shitov2021column}&
 \cite[Theorem 1.1]{civril2014column}$^{(2)}$ &
$\rho_F(\mc)\equiv \|(\mi-\mc\mc^{\dagger})\ma\|_F$\\
\hline
\end{tabular}
}
\caption{Complexity of criteria for 
selecting \changes{$k\leq \min\{\rank(\ma), m/3\}$ representative columns $\mc\in\real^{m\times k}$ from $\ma\in\rmn$:}
(1) Exponential inapproximability of the volume is shown
in \cite[Theorem 5]{CMI2013}.
(2) Minimizing the Frobenius norm residual is
 NP complete \cite[Theorem 2.2]{shitov2021column}, and 
is PTAS inapproximable conditional on the \textit{Unique Games
 Conjecture} being true \cite[Theorem 1.1]{civril2014column}.}
 \end{sidewaystable}

\begin{comment}
\begin{example}[Examples for criteria in Table~\ref{t_intro}]\label{ex_intro}
Let $k=2$ and
\begin{align*}
\ma=\begin{bmatrix}20&&& \\ 
&10&&\\ &&1& \\ &&&1\end{bmatrix}\in\real^{4\times 4}.
\end{align*}
The absolute two-norm criteria select the 
leading 2 columns of $\ma$, which are
\begin{align*}
\mc_*=\begin{bmatrix} 20& 0&0&0\\ 
0&10&0&0\end{bmatrix}^T\in\real^{4\times 2}
\end{align*}
and have
\begin{equation*}
\vol(\mc_*)=20\cdot 10=200, \quad
\|\mc_*\|_2=20, \quad
\|\mc_*^{\dagger}\|_2=\frac{1}{10},\quad
\|(\mi-\mc_*\mc_*^{\dagger})\ma\|_2=1.
\end{equation*}
In contrast, the relative criteria select the
trailing two columns of~$\ma$, which are
\begin{equation*}
\mc_*= \begin{bmatrix} 0& 0 &1&0\\
0&0&0&1\end{bmatrix}^T\in\real^{4\times 2}
\end{equation*}
and have 
\begin{equation*}
\rvol(\mc_*)=\frac{1\cdot 1}{1^2}=1,\quad
\kappa_2(\mc)=\frac{1}{1}=1,\quad
\sr(\mc_*)=\frac{2}{1}=2, \quad \sr(\mc_*^{\dagger})=2.
\end{equation*}
Since $\ma$ is a diagonal matrix, the absolute values of its diagonal elements are equal to the respective 
columns norms, so that $\sopt(\mc_*)=1$ for any submatrix $\mc_*\in\real^{4\times 2}$ of~$\ma$.
\end{example}
\end{comment}

\begin{comment}
Criterion~4 selects the trailing $n-k+1$ columns of $\ma$
with 
\begin{align*}
\sigma_1(\mc_*)=a_{n-k+1,n-k+1}.
\end{align*}
Criterion~2 selects the $k$ columns $\mc_*$ with
\begin{align*}
\kappa(\mc_*)=\min_{1\leq j\leq n-k+1}{\frac{a_{jj}}{a_{k+j-1,k+j-1}}}.
\end{align*}
\end{comment}

\subsection{Overview}
We derive NP hardness results and PTAS approximation thresholds for:
relative volume maximization (Section~\ref{s_rvol}); 
volume and S-optimality maximization (Section~\ref{s_volume}); 
two-norm, Frobenius norm and Schatten $p$-norm minimization (Section~\ref{s_twop}); two-norm,
Frobenius norm and Schatten $p$-norm minimization of pseudo-inverses (Section~\ref{s_psnorm});
condition number minimization (Section~\ref{s_condno}); and stable rank maximization (Section~\ref{s_srank}). 
To formulate the subset selection criteria as decision problems,
we derive optimal values for the criteria  (Section~\ref{s_unitnorm}), 
as well as expressions for partitioned pseudo-inverses
(Section~\ref{s_partinv}).

\subsection{Contributions}

Our main contributions include the following:
\begin{enumerate}
\item We introduce the 
criterion of `relative volume maximization'
and show that it is NP-hard and does not admit
a PTAS. \changes{Unlike the volume, the relative volume is able to detect whether
a matrix is ill-conditioned, because it is inversely proportional to a power of the condition number (Section~\ref{s_rvol}).}
\item We present rigorous proofs to derive PTAS inapproximability thresholds  
that differ from those in \cite{CM09} for
two-norm minimization of the pseudo-inverse
(Section~\ref{s_psnorm2})
and
two-norm condition number minimization (Section~\ref{s_condno2}).

\item We show that the following criteria are NP-hard:
$p$-norm pseudo-inverse minimization,
$p$-norm condition number minimization, and mixed $p$-norm condition number minimization for $p\geq 2$ (Section~\ref{s_condno1}), as well as stable rank 
maximization and $p$-stable rank maximization for $p>2$
(Section~\ref{s_srank1}).

\item We show that the following criteria do not admit 
a PTAS:
Frobenius norm minimization of the pseudo-inverse (Section~\ref{s_psnorm3}) and $p$-norm minimization of the 
pseudo-inverse for $p>2$ (Section~\ref{s_psnorm4});
Frobenius norm condition number minimization (Section~\ref{s_condno3}) and mixed  
condition number minimization (Section~\ref{s_condno4});
and
stable rank maximization (Section~\ref{s_srank2}).

\item We derive optimal values \changes{for} the subset
selection criteria (Section~\ref{s_unitnorm}), so that 
their optimization versions can be formulated as decision
problems.

\item We derive expressions and $p$-norm bounds for 
partitioned pseudo-inverses (Section~\ref{s_partinv}).
\end{enumerate}

%\item We present an algorithm for \emph{local} relative 
%volume maximization (\ref{e_relvol}), which 
%is a relative version of the sRRQR algorithm
%in \cite{GuEis96}, and show that it produces well-conditioned columns (Section~\ref{s_maxvol}).
%\item 
%Numerical experiments  illustrate that 
%relative volume maximization tends to require fewer iterations than
%condition number maximization, while producing 
%columns that are as well-conditioned
%(Section~\ref{s_exp1}).
%\item We present fast algorithms  for relative volume maximization based on randomized sketching and 
%rank-one updating (Section~\ref{s_speed}).

\subsection{\changes{Roadmap}}
\changes{We give a brief sketch of our approaches for deriving NP hardness results
(Section~\ref{s_141}) and PTAS results (Section~\ref{s_142}) for the above subset selection
criteria.}

\changes{\subsubsection{Extension of the NP-hardness results from \cite{CM09}}\label{s_141}
We perform the following steps.
\begin{enumerate}
\item Polynomial time reduction of the subset selection criteria to X3C by representing the collection of subsets as a matrix $\ma$ with unit norm columns.  
\item For such matrices with unit-norm columns, we establish optimal values for the subset selection criteria (Section~\ref{s_unitnorm}), and show that the optimal values are achieved by matrices with  
orthonormal columns. \\ 
Optimal values are derived for the criteria:
maximal volume and S-optimality (Lemma~\ref{l_vol}),
maximal relative volume (Lemma~\ref{lem:orth}), minimal norms (Lemma~\ref{l_norm}), 
minimal pseudo-inverse norms (Lemma~\ref{l_pinv}),
minimal condition numbers (Lemma~\ref{l_cond}), and maximal stable ranks (Lemma~\ref{l_srank}).
\item According to \cite{CM09}, X3C is true is and only if $\ma$ has a submatrix with orthonormal columns. 
\item Combining the two previous items implies that the criterion is optimal if and only if X3C is true.\\
This shows NP hardness of the criteria: relative volume maximization (Theorem~\ref{thm:rvolnp}), volume and S-optimality maximization (Theorem~\ref{t_voldec}), norm minimization (Theorem~\ref{t_norm}), pseudo-inverse norm minimization
(Theorem~\ref{t_pi7}), condition number minimization (Theorem~\ref{t_cond}), and stable rank maximization (Theorem~\ref{t_other}).
\end{enumerate}
}

\changes{
\subsubsection{Systematic approach for extending PTAS results from \cite{CM09}}\label{s_142}
We perform the following steps.
\begin{enumerate}
\item Monotonicity: Singular value interlacing implies that deleting columns from a matrix can only increase its relative volume (Lemma~\ref{l_inter}) and decrease the condition number (Lemma~\ref{l_inter2}). In other words, adding columns can only decrease the relative volume and increase
the condition number.
\item It therefore suffices to show the lack of a PTAS 
by establishing approximability thresholds  for matrices with 2 columns. 
\item Monotonicity implies that matrices with more columns cannot cross the approximability thresholds.\\
This shows that the following criteria do not admit a PTAS: maximization of relative volume (Theorem~\ref{t_volptas}), volume (Theorem~\ref{t_vol}) and S-optimality (Theorem~\ref{t_sopt}); two-norm minimization (Theorem~\ref{t_2m}); pseudo-inverse minimization 
in the two- (Theorem~\ref{t_ssv}), Frobenius (Theorem~\ref{t_pi}) and Schatten $p$-norms (Theorem~\ref{t_pi3});
condition number minimization in the two- (Theorem~\ref{t_2cond}), Frobenius (Theorem~\ref{t_fcond}) and mixed norms
(Theorem~\ref{t_cd}); and maximization of the stable rank (Theorem~\ref{t_sr}).
\end{enumerate}
}

\subsection{Literature}\label{s_lit}
We review existing work on column subset selection, \changes{focused} on computational complexity.
Early surveys of column subset
selection algorithms include \cite{CI1994,HongPan92}.

\paragraph{Norm minimization}
NP hardness and PTAS inapproximability for the two-norm are shown in
\cite[Theorems 5 and~7]{CM09}.

\paragraph{Pseudo-inverse norm minimization}
NP hardness and PTAS inapproximability are shown in
\cite[Theorems 4 and~7]{CM09} for the two-norm.
Randomized and deterministic 
algorithms for wide matrices $\ma\in\rmn$ with $m<n$
are given in \changes{\cite{AD2025,AvronBout2013,KO2025,NikoST19,NikoST22,Osinsky2024}}.

\changes{In the context of A-optimal design, trace minimization corresponds to 
minimization of the pseudo-inverse in the Frobenius norm \cite{NikoST19,NikoST22}. 
For instance, if $A$ has full row rank
then $\ma^{\dagger}=\ma^T(\ma\ma^T)^{-1}$ and
\begin{align*}
\|\ma^{\dagger}\|_F^2=\trace\left((\ma^{\dagger})^T\ma^{\dagger}\right)=\trace\left((\ma\ma^T)^{-1}\right).
\end{align*}
The reduction for $m=k$ is based on \textit{Odd Cycle Packing}, which consists of finding a maximal family of vertex-disjoint odd cycles in a simple undirected graph.}

\paragraph{Condition number minimization}
NP hardness and PTAS inapproximability are shown in
\cite[Theorems 4 and~8]{CM09} for two-norm condition numbers.

\paragraph{Volume maximization}
NP hardness and PTAS inapproximability are shown in
\cite[Theorems 4 and~6]{CM09}, while the exponential inapproximability of the volume is presented in \cite[Theorem 5]{CMI2013},
\changes{\cite[Theorem 8]{DEFM15}.
The latter involves a reduction based on \textit{Odd Cycle Packing}
for matrices $\ma\in\rmn$ with $\rank(\ma)=m$, $k = m$, $m=1320d$ and $n=2112 d$ for some $d>0$.}

\changes{An early result \cite[Theorem 3]{Welch1982}
in the context of D-optimality
shows the NP-hardness of volume maximization via a reduction from \textit{Hamiltonian Circuit}, where
$k\geq n$ rows $\mc\in\real^{k\times n}$ from an $m\times n$ adjacency matrix are to be selected to maximize $\vol(\mc)^2=\det(\mc^T\mc)$. 
%Here $m$ does not need to be a multiple of $k$.
}

Greedy algorithms are presented in~\cite{CM09,dHM2011}, and
local deterministic and randomized algorithms in \changes{\cite{GZ2025,GuEis96,NikoS16,Osinsky2023b}}. This is extended to Hilbert spaces in \cite{KressNi2024}.
The importance of 
local volume maximization for revealing the rank is demonstrated
in~\cite{DGTY2024}.

The S-optimality criterion bears a resemblance to the relative volume and was introduced in 
\cite[(3.5)]{SX2016} and \cite[(3.10)]{LCS2024}.

\paragraph{Residual Minimization}
The NP completeness in the Frobenius norm is established 
in \cite[Theorem 2.2]{shitov2021column}, and PTAS inapproximability conditional on the \textit{Unique Games
 Conjecture} being true \cite[Theorem 1.1]{civril2014column}.
Algorithms for the two- and Frobenius norms are given
in \cite{BDMI2014}.

\changes{Early randomized algorithms, based on volume sampling are given in~\cite{BDMI2014,DeshR10,DeshR06}. A statistical 
perspective is provided in \cite{SoodHastie2025}, and it is shown that subset selection is equivalent to the method
of principal variables.}

\subsection{Assumptions}
Let $\ma\in\rmn$ be a general matrix with $\rank(\ma)\geq k$, 
and singular values 
$\sigma_1(\ma)\geq \cdots\geq\sigma_{\min\{m,n\}}(\ma)\geq 0$.
Let $\mc\in\real^{m\times k}$ 
be a submatrix of $k$ columns of $\ma$, with singular values 
\begin{equation}\label{e_so}
\sigma_1(\mc)\geq \cdots\geq\sigma_k(\mc)\geq 0.
\end{equation}

Singular value interlacing \cite[7.3.P44]{HoJoI}
implies
\begin{equation}\label{e_inter}
\sigma_{\min\{m,n\}-k+j}(\ma)\leq \sigma_j(\mc)\leq
    \sigma_j(\ma),\qquad 1\leq j\leq k.
\end{equation}
We denote by $\mi_n\equiv\begin{bmatrix} \ve_1& \cdots &\ve_n\end{bmatrix}\in\rnn$ the identity matrix, and by $\ve_j\in\rn$ its columns.
\section{Relative Volume}\label{s_rvol}
We introduce the concept of maximal volume, and show that relative volume maximization is NP hard 
(Section~\ref{s_rvol1}) and admits no PTAS (Section~\ref{s_rvol2}).

The relative volume of a matrix $\mc\in\real^{m\times k}$ with $\rank(\mc)=k$, defined as 
\begin{equation*}
\rvol(\mc) \equiv\frac{\vol(\mc)}{\|\mc\|_2^k}=
\prod_{j=1}^k{\frac{\sigma_j(\mc)}{\sigma_1(\mc)}},
\end{equation*}
where the volume of a matrix $\mc\in\real^{m\times k}$
with $\rank(\mc)=k$ is defined as
\begin{equation}\label{e_vol}
\vol(\mc)\equiv\sqrt{\det(\mc^T\mc)}=
\prod_{j=1}^k{\sigma_j(\mc)}.
\end{equation}
If $\mc$ is square, then $\vol(\mc)=|\det(\mc)|$.

Geometrically, the relative volume is the ratio of the  parallelepiped volume to the volume of the cube that contains the parallelepiped.

\paragraph{\changes{Example}}
\changes{In contrast to the volume, the relative volume has the advantage of being able to detect when a matrix 
is ill-conditioned. For instance, the matrix 
\begin{align*}
\ma=\begin{bmatrix}1/\epsilon & 0 \\ 0 & \epsilon \end{bmatrix}\qquad \text{where}\quad 0<\epsilon\ll 1
\end{align*}
is ill-conditioned with condition number $\kappa_2(\ma)=1/\epsilon^2$. However, the volume  does not detect this,
since $\vol(\ma)=1$,
while the relative volume does, $\rvol(\ma)=\epsilon^2=1/\kappa_2(\mc)$.}

\changes{
\begin{remark}\label{r_1}
The relative volume is inversely proportional
to some power of the two-norm condition number. Let $\mc\in\real^{m\times k}$ with $\rank(\mc)=k$. Then 
\begin{align*}
\left(\frac{1}{\kappa_2(\mc)}\right)^{k-1}\leq \rvol(\mc)\leq \frac{1}{\kappa_2(\mc)}.
\end{align*}
\end{remark}
}

Next, we want to find $k$ columns $\mc$ of $\ma\in\rmn$ 
with maximal relative volume,
$\max_{\mc}{\rvol(\mc)}$.

\begin{comment}
We first show
that only matrices with orthonormal columns have the maximal relative volume~1.

\begin{lemma}\label{lem:orth} 
Let $\mc\in \mathbb{R}^{m \times k}$ with $\rank(\mc)=k$. Then $0<\rvol(\mc) \le 1$.

Furthermore, $\rvol(\mc) = 1$  if and only if the columns of $\mc$ are orthonormal.
\end{lemma}
\begin{proof}
The inequalities follows from the singular
value ordering~(\ref{e_so}).

As for the equality, if $\mc$ has orthonormal columns, then all singular values of $\mc$ are
equal to~1,
and $\rvol(\mc) = 1$. Conversely, if $\rvol(\mc) = 1$ then 
\begin{equation}\label{e_ortho}
\prod_{j=1}^k{\frac{\sigma_j(\mc)}{\sigma_1(\mc)}} = 1.
\end{equation}
Since all factors $\sigma_j(\mc)/\sigma_1(\mc)\leq 1$, (\ref{e_ortho}) can only 
hold if all factors $\sigma_j(\mc)/\sigma_1(\mc)=1$.
Hence all singular values of $\mc$ are equal to~1, and $\mc$ has orthonormal columns.
 \end{proof}
\end{comment}

\subsection{Decision problem}\label{s_rvol1}
Following \cite{Garey}, we state the decision problem for the maximal relative volume as follows:
%\aks{Should $\ma \in \mathbb{Q}^{m\times n}$ and $p\ge 1$ rational, here and elsewhere?}
%\ii{Although the matrices to which we are reducing are multiples of integers, they not rational, what with the square root. This assumption on the matrix elements would need to be stated in Section 1.} \aks{I added a comment below in the proof. }

\begin{quote}
\emph{Instance:} Matrix $\ma \in \real^{m\times n}$ with unit-norm columns 
$\|\ma\ve_j\|_2=1$, $1\leq j\leq n$;
\changes{integer $1\leq k\leq \min\{\rank(\ma), m/3\}$};
and parameter $0< b\leq 1$.\\
\emph{Question:} Is there a column submatrix $\mc\in \mathbb{R}^{m\times k}$ of $\ma$ such that 
\begin{equation}\label{eqn:rvolreform} 
\rvol(\mc)\geq b. 
\end{equation}
\end{quote}

\begin{theorem}\label{thm:rvolnp}
\changes{The decision problem of} relative volume maximization is NP hard.
\end{theorem}

\begin{proof}
Following the proof of~\cite[Proof of Theorem 4]{CM09},
we perform a 
polynomial-time reduction from the 
NP-complete problem \textit{Exact Cover}
\textit{by 3-Sets (X3C)} \cite[Section A.3.1]{Garey} to relative volume maximization.
\begin{quote}
\emph{Instance:}
Set $\mathcal{S} = \{1,\dots,3M\}$ and collection of sets $\mathcal{K} = \{\mathcal{K}_1,\dots,\mathcal{K}_n\}$,
where each $\mathcal{K}_i\subset\mathcal{S}$
has cardinality $3$.\\
\emph{Question:} Is there a subset $\mathcal{K}'\subset \mathcal{K}$ that forms an \textit{exact cover}\footnote{An \textit{exact cover} means that every element of $\mathcal{S}$ appears exactly once in $\mathcal{K}'$.} for~$\mathcal{S}$?
\end{quote}

Assume that
 $\mathcal{S} = \{1,\dots,3M\}$ and a collection of 
 distinct subsets $\mathcal{K} = \{\mathcal{K}_1,\dots,\mathcal{K}_n\}$
 represent an instance of X3C.
We show that this instance  can be solved if we can construct a matrix $\ma$ with a submatrix~$\mc_*$ that 
solves~\eqref{eqn:rvolreform}. 

To this end, let $\ma \in \mathbb{R}^{(3M) \times n}$  
be a matrix where each column corresponds to a 
3-element set $\mathcal{K}_j$, and the non-zero elements in the column
represent the elements of~$\mathcal{K}_j$,
\begin{equation}\label{e_X3C}
a_{ij} = \begin{cases}\frac{1}{\sqrt{3}} & \text{if $i \in \mathcal{K}_j$}  \\ 0 & \text{otherwise}\end{cases},
\qquad 1 \le i \le 3M, \ 1\le j \le n.
\end{equation}
Each column of $\ma$ has exactly 3 non-zero elements
and $\|\ma\ve_j\|_2=1$, $1\leq j\leq n$.
In the reduced instance we set
$m = 3M$, $k=M$, and $b=1$. \changes{To construct matrices $\ma$ with more than $3k$ rows, just append 
the appropriate number of zero rows. This does not change the singular values of $\ma$.}
Clearly, constructing $\ma$ requires polynomial time.

%\aks{Although the matrix $\ma$, thus constructed, has irrational entries, we can work with the field $\mathbb{Q}_3 = \{ a + b\sqrt{\tfrac{1}{3}}| a, b \in \mathbb{Q}\}$, in which each number $a + b/\sqrt{3}$ can be represented by the rationals $a,b,3$ and elementary operations can be done in $\mathbb{Q}$. }

According to~\cite[Proof of Theorem 4]{CM09}, the instance of X3C is true if and only if there exists a submatrix $\mc_*\in \mathbb{R}^{(3M)\times M}$ with orthogonal columns. Lemma~\ref{lem:orth} \changes{in Section~\ref{s_unitnorm2}} implies that  this is equivalent to $\rvol(\mc_*) = 1$. 
\changes{Due to this equivalence and the structure of~$\ma$, determining whether $\rvol(\mc_*)=1$ 
simply amounts to finding orthogonal
columns based on the sparsity pattern of $\mc_*$.} Therefore, the instance of X3C is true if and only if $\rvol(\mc_*) = 1$. Thus, we have constructed a polynomial time reduction from an NP complete problem, establishing that relative volume maximization is NP hard.   %being an optimizer of~\eqref{eqn:rvolreform}. Therefore, the instance of X3C is true if and only if $\max_{\mc\in \mathbb{R}^{3M\times m}} \rvol(\mc) = 1$. 
\end{proof}

\subsection{No PTAS for relative volume maximization}\label{s_rvol2}
Lemma~\ref{lem:orth} implies that the maximal relative  
volume equals~1. We derive an inapproximability threshold of~$\sqrt{1/2}$.

\begin{theorem}\label{t_volptas}
It is NP hard to approximate relative volume maximization within a factor of 
$1/\sqrt{2}$. Thus, unless P=NP, there is no PTAS for relative volume maximization.
%
%There is no PTAS that can find, for all $\epsilon>0$,
% a submatrix $\mc_*\in\real^{m\times k}$ of $\ma$ so that 
%\begin{equation*}
%\rvol(\mc_*)\geq \sqrt{1/2}+\epsilon
%\end{equation*} 
%unless $P=NP$.
\end{theorem}
%\aks{Question: is the factor defined as true/approximate or its inverse? }
%\ii{Sorry, typo},

\begin{proof}
This is analogous to the proof of \cite[Theorem 5]{CM09} for volume maximization. We derive an upper bound for $\rvol(\mc)$
for the reduced instance of X3C when it is false.
Lemma~\ref{l_inter} \changes{in Section~\ref{s_unitnorm2}} implies that the relative
volume increases with the removal of columns.
Hence, the largest relative volume is determined by finding
two columns for a false instance of X3C. 

Let $\ma$ be constructed as in (\ref{e_X3C}),
and 
assume that the X3C instance is false. Then any collection $\mathcal{K}$ of cardinality $k=M$ has at least two sets $\mathcal{K}_i$ and $\mathcal{K}_j$ with $i\neq j$ and a non-empty intersection. Since the subsets $\mathcal{K}_j$ are distinct,
the overlap consists of 1 or of 2 elements. Let $\mc\in\real^{(3M)\times M}$ be a 
column submatrix of $\ma\in\real^{(3M)\times n}$
that represents such a collection.
 
\begin{enumerate}
\item Consider the case where the two sets share
a single element, that is,
$|\mathcal{K}_i\cap\mathcal{K}_j|=1$. 
Let the sets be represented by 
columns $\vc_i\equiv\mc\ve_i$ and $\vc_j\equiv \mc\ve_j$. 
The non-zero elements of the submatrix 
$\hat{\mc}\equiv\begin{bmatrix}\vc_i & \vc_j\end{bmatrix}$ are, up to row permutations,
\begin{equation}\label{e_aux1}
\begin{bmatrix} 1/\sqrt{3} & 1/\sqrt{3} & 1/\sqrt{3} & 0 & 0\\ 0 &0 & 1/\sqrt{3} & 1/\sqrt{3}&1/\sqrt{3}
\end{bmatrix}^T\in\real^{5\times 2}.
\end{equation}
With appropriate orthogonal row rotations this matrix can be further reduced to 
\begin{equation*}
\begin{bmatrix} 0 & 1 & 0& 0 & 0\\ 0 &1/3 & \sqrt{8}/3 & 0 &0
\end{bmatrix}^T.
\end{equation*}
Thus $\hat{\mc}$ has the same singular values as the matrix
\begin{equation*}
\begin{bmatrix} 1& 1/3\\ 0 &\sqrt{8}/3\end{bmatrix},
\end{equation*}
whose singular values are $2/\sqrt{3}$ and $\sqrt{2/3}$.
Hence,
$\sigma_1(\hat{\mc})=2/\sqrt{3}$
and $\sigma_2(\hat{\mc})=\sqrt{2/3}$.
Lemma~\ref{l_inter} \changes{in Section~\ref{s_unitnorm2}} implies
\begin{equation*}
\rvol(\mc)\leq \rvol(\hat{\mc})=\sigma_2(\hat{\mc})/\sigma_1(\hat{\mc})=1/\sqrt{2}.
\end{equation*}

\item Now consider the case where the sets share
two elements. That is, if
$|\mathcal{K}_i\cap\mathcal{K}_j|=2$, 
we can find two columns 
$\hat{\mc}\equiv\begin{bmatrix}\vc_i & \vc_j\end{bmatrix}$ of $\mc$ whose non-zero
elements are, up to row permutations,
\begin{equation}\label{e_aux2}
\begin{bmatrix} 1/\sqrt{3} & 1/\sqrt{3} & 1/\sqrt{3} & 0\\ 0 &1/\sqrt{3} & 1/\sqrt{3} & 1/\sqrt{3}
\end{bmatrix}^T\in\real^{4\times 2}.
\end{equation}
After row rotations, the nonzero elements are
\begin{equation*}
\begin{bmatrix} 1& 2/3\\ 0 &\sqrt{5}/3\end{bmatrix},
\end{equation*}
whose singular values are $\sqrt{5/3}$ and $\sqrt{1/3}$.
Hence,
$\sigma_1(\hat{\mc})=\sqrt{5/3}$
and $\sigma_2(\hat{\mc})=\sqrt{1/3}$.
Lemma~\ref{l_inter} \changes{in Section~\ref{s_unitnorm2}} implies
\begin{equation*}
\rvol(\mc)\leq \rvol(\hat{\mc})=\sigma_2(\hat{\mc})/\sigma_1(\hat{\mc})=1/\sqrt{5}.
\end{equation*}
Thus, the relative 
volume of (\ref{e_aux2}) is smaller than the one of (\ref{e_aux1}).
\end{enumerate}

Considering the overlap of only two sets is sufficient for the following reason. 
Suppose that $\ell< k$ columns of $\mc$ overlap,
and they are represented by the submatrix~$\mc_{\ell}$.
Lemma~\ref{l_inter} \changes{in Section~\ref{s_unitnorm2}} implies that $\rvol(\mc)\leq \rvol(\mc_{\ell})$. Let $\mc_2$ be any 2-column submatrix
of~$\mc_{\ell}$.
Again, Lemma~\ref{l_inter} 
implies $\rvol(\mc)\leq \rvol(\mc_{\ell})\leq \rvol(\mc_2)$.
The non-zero elements of $\mc_2$ can only be of the form (\ref{e_aux1})
or~(\ref{e_aux2}). Between the two, (\ref{e_aux1}) achieves 
the largest relative volume $1/\sqrt{2}$, which implies
$\rvol(\mc)\leq 1/\sqrt{2}$.

Summary: As a consequence, it is NP hard to approximate the relative volume by a factor of $1/\sqrt{2}$. If an algorithm for relative volume maximization did indeed return a set of columns $\mc$ with
$\rvol(\mc)>1/\sqrt{2}$, then this instance of X3C would be true. Thus, relative volume maximization could distinguish between true and false instances of X3C,  implying $P=NP$.
\end{proof}

\section{Volume and S-optimality}\label{s_volume}
We show that volume
and S-optimality maximization are NP \changes{hard} (Section~\ref{s_volume1}),
give a rigorous proof of the PTAS inapproximability threshold 
for volume maximization (Section~\ref{s_volume2}), and 
show that S-optimality maximization does not admit a
PTAS (Section~\ref{s_volume3}).

We want to find $k$ columns~$\mc$ of $\ma\in\rmn$ with maximal
volume, $\max_{\mc}{\vol(\mc)}$.
This criterion is NP hard 
and does not admit a PTAS \cite{CM09}.
The best deterministic algorithm~\cite{GuEis96}
is a local volume maximization that produces a matrix~$\mc$ whose singular values are maximal subject to pairwise permutations of columns. 

The S-optimality of a matrix $\mc\in\real^{m\times k}$
with $\rank(\mc)=k$ is defined as
\cite[(3.5)]{SX2016} and \cite[(3.10)]{LCS2024} 
\begin{equation*}
\sopt(\mc)\equiv \left(\frac{\vol(\mc)}{\prod_{i=1}^{k}{\|\mc\ve_i\|_2}}\right)^{1/k}
=\left(\prod_{i=1}^k{\frac{\sigma_i(\mc)}{\|\mc\ve_i\|_2}}\right)^{1/k},
\end{equation*}
and greedy algorithms are presented for maximizing the S-optimality.
We want to find $k$ columns~$\mc$ of $\ma\in\rmn$ with maximal S-optimality,
$\max_{\mc}{\sopt(\mc)}$. 

\subsection{Decision problems}\label{s_volume1}
Given a matrix $\ma\in\rmn$ with unit-norm columns $\|\ma\ve_j\|_2=1$, $1\leq j\leq n$;
\changes{integer $1\leq k\leq \min\{\rank(\ma), m/3\}$};
and parameter $0\leq b\leq 1$.
\begin{enumerate}
\item Volume maximization:\\
Does $\ma$ have a submatrix $\mc\in\real^{m\times k}$ with 
$\vol(\mc)\geq b$?
\item S-optimality maximization:\\Does $\ma$ have a submatrix $\mc\in\real^{m\times k}$ with 
$\sopt(\mc)\geq b$?
\end{enumerate}
The NP hardness of volume maximization was already established in~\cite[Theorem 4]{CM09}. 

\begin{theorem}\label{t_voldec}
The decision problems: volume maximization and S-optimality maximization are NP hard. 
\end{theorem}

\begin{proof}
The idea is that the reduction to X3C in Theorem~\ref{thm:rvolnp} only requires orthonormality:
The instance of X3C is true if and only if $\ma$ has a submatrix $\mc\in \mathbb{R}^{(3M)\times M}$ with orthonormal columns. 

As in~(\ref{e_X3C}), we construct the matrix~$\ma\in\real^{3M\times n}$
where each column has exactly~3 non-zero elements and $\|\ma\ve_j\|_2=1$, $1\leq j\leq n$.
According to~\cite[Proof of Theorem 4]{CM09}, the instance of X3C is true if and only if $\ma$ has a submatrix $\mc\in \mathbb{R}^{(3M)\times M}$ with orthonormal columns. Lemma~\ref{l_vol} \changes{in Section~\ref{s_unitnorm1}} implies that  this is equivalent to $\mc$ satisfying the equalities $\vol(\mc) = 1$ or $\sopt(\mc) = 1$. Thus, volume maximization and S-optimality maximization are NP hard.

%Since it takes only polynomial time to verify if $\vol(\mc) = 1$ or $\sopt(\mc) = 1$, both problems are in NP and, therefore, are NP complete. 
\end{proof}

\subsection{No PTAS for volume maximization}\label{s_volume2}
Lemma~\ref{l_vol} \changes{in Section~\ref{s_unitnorm1}} implies that a matrix 
with unit-norm columns has maximal volume equal to~1.
We derive an inapproximability threshold of $2\sqrt{2}/3$.

\begin{theorem}\label{t_vol}
It is NP hard to approximate volume maximization within a factor of 
$2\sqrt{2}/3$. Thus, unless P=NP, there is no PTAS for volume maximization.
%
%There is no PTAS that can find, for all $\epsilon>0$,
%a submatrix $\mc_*\in\real^{m\times k}$ of $\ma$ so that 
%\begin{equation*}
%\vol(\mc_*) \geq 2\sqrt{2}/3+\epsilon 
%\end{equation*}
%unless $P=NP$.
\end{theorem}

\begin{proof}
The setup is the same as in the proof of Theorem~\ref{t_volptas}.
Assume that the X3C instance is false.

The proof is by induction over the number of columns $k$. 
\begin{itemize}
\item For $k=2$, two distinct sets can overlap in one or in two elements. 
When the sets share a single element, the nonzero elements (\ref{e_aux1}) of the relevant two columns~$\hat{\mc}_2$ 
have singular values 
$\sigma_1(\hat{\mc}_2)=2/\sqrt{3}$ and $\sigma_2(\hat{\mc}_2)=\sqrt{3/2}$.
When the sets share two elements, the nonzero elements
(\ref{e_aux2}) of the relevant two columns $\hat{\mc}_2$
have singular values $\sigma_1(\hat{\mc}_2)=\sqrt{5/3}$ and
$\sigma_2(\hat{\mc}_2)=\sqrt{1/3}$.
Between the two, (\ref{e_aux1}) achieves the largest volume.
Hence, for matrices $\hat{\mc}_2$ with $k=2$ columns, (\ref{e_aux1}) implies that
\begin{equation*}
\vol(\hat{\mc}_2)=\sigma_1(\hat{\mc}_2)\sigma_2(\hat{\mc}_2)
= \tfrac{2}{\sqrt{3}}\sqrt{\tfrac{2}{3}}=\tfrac{2\sqrt{2}}{3}.
\end{equation*}

\item For matrices $\hat{\mc}_{k-1}$ with $k-1$ columns, assume that 
\begin{equation*}
\vol(\hat{\mc}_{k-1})\leq \tfrac{2\sqrt{2}}{3}.
\end{equation*}

\item Let $\hat{\mc}_k\equiv
\begin{bmatrix} \hat{\mc}_{k-1} & \vc\end{bmatrix}$
be a matrix with $k$ columns, where $\|\vc\|_2=1$. 
The Gram matrix is
\begin{equation*}
\hat{\mc}_k^T\hat{\mc}_k=\begin{bmatrix}\hat{\mc}_{k-1}^T\hat{\mc}_{k-1} & \hat{\mc}_{k-1}^T\vc\\
\vc^T\hat{\mc}_{k-1} & \vc^T\vc\end{bmatrix}
\end{equation*}
and \cite[(2)]{Cot74} implies
%where
%\begin{align*}
%\ml&\equiv\begin{bmatrix}\mi & \vzero\\ 
%\vc^T\hat{\mc}_{k-1}(\hat{\mc}_{k-1}^T\hat{\mc}_{k-1})^{-1} & 1\end{bmatrix}\\
%\mU&\equiv\begin{bmatrix}\hat{\mc}_{k-1}^T\hat{\mc}_{k-1} & \hat{\mc}_{k-1}^T\vc\\ \vzero&
%\vc^T\vc-\vc^T\hat{\mc}_{k-1}(\hat{\mc}_{k-1}^T\hat{\mc}_{k-1})^{-1}\hat{\mc}_{k-1}^T\vc\end{bmatrix}.
%\end{align*}
%From 
%\begin{equation*}
%\vol(\hat{\mc}_k)^2=\det(\hat{\mc}_k^T\hat{\mc}_k)=
%\det(\ml\mU)=\det(\ml)\det(\mU)=\det(\mU)
%\end{equation*}
%follows
\begin{equation}\label{e_sc}
\vol(\hat{\mc}_k)^2=\det(\hat{\mc}_{k-1}^T\hat{\mc}_{k-1})
\det(\vc^T\vc-\vc^T\hat{\mc}_k^T\hat{\mc}_k^{\dagger}\vc).
\end{equation}
Since the second factor is a scalar and contains the orthogonal projector $\mP\equiv \mi-\hat{\mc}_k^T\hat{\mc}_k^{\dagger}$, we can write the second factor as
\begin{equation*}
\det(\vc^T\vc-\vc^T\hat{\mc}_k^T\hat{\mc}_k^{\dagger}\vc)=\vc^T\mP\vc=\|\mP\vc\|_2^2\leq \|\vc\|_2^2=1.
\end{equation*}
Inserting this into (\ref{e_sc}) gives
\begin{equation*}
\vol(\hat{\mc}_k)^2\leq \det(\hat{\mc}_{k-1}^T\hat{\mc}_{k-1})=\vol(\hat{\mc}_{k-1})^2\leq 
\left(\tfrac{2\sqrt{2}}{3}\right)^2.
\end{equation*}
\end{itemize}
As in the proof of Theorem~\ref{t_volptas}, we conclude that it is NP hard to approximate the maximal volume within a factor of $2\sqrt{2}/3$. Thus, there is no PTAS unless P=NP.
\end{proof}

Theorem~\ref{t_vol} presents a rigorous proof for the inapproximability threshold 
from~\cite[Theorem 6]{CM09}. 
The exponential inapproximability of the volume is presented in \cite[Theorem 5]{CMI2013}, in the sense that there exists $0<\delta<1$ and $c>0$ such that volume maximization is not
approximable with $2^{-ck}$ for $k=\delta n$ unless $P=NP$.
However, the estimation of the
volume from inner products of selected vectors
in \cite[Proof of Lemma 16]{CMI2013} is not clear.

\subsection{No PTAS for S-optimality maximization}\label{s_volume3}
Lemma~\ref{l_vol} \changes{in Section~\ref{s_unitnorm1}}
implies that a matrix with unit-norm columns has maximal
S-optimality of~$1$. 
We derive an inapproximability threshold of $(2\sqrt{2}/3)^{1/k}$.

\begin{theorem}\label{t_sopt}
It is NP hard to approximate S-optimality maximization within a factor of 
$(2\sqrt{2}/3)^{1/k}$. Thus, unless P=NP, there is no PTAS for S-optimality maximization.
%There is no PTAS that can find, for all $\epsilon>0$,
%a submatrix $\mc\in\real^{m\times k}$ of $\ma$ so that 
%\begin{equation*}
%\sopt(\mc) \geq (2\sqrt{2}/3)^{1/k}+\epsilon 
%\end{equation*}
%unless $P=NP$.
\end{theorem}

\begin{proof}
This follows from the assumption of $\mc$ having unit-norm
columns, so that 
$\sopt(\ms)=\vol(\mc)^{1/k}$, and from Theorem~\ref{t_vol}.
\end{proof}

\section{Two-norm and Schatten p-norms}\label{s_twop}
We show  that minimization of the two-norm and Schatten $p$-norms for $p>2$ is
NP hard (Section~\ref{s_twop1}), and
that there is no PTAS for two-norm minimization (Section~\ref{s_twop2}).
%and for Schatten $p$-norm minimization (Section~\ref{s_twop3}).

For $p\geq 1$, the Schatten $p$-norm of a matrix $\ma\in\rmn$ is defined as 
\begin{equation}%\label{e_schatten_def}
   \|\ma\|_{(p)} \equiv \left(\sum_{j=1}^{\min\{m,n\}} \sigma_j(\ma)^p\right)^{1/p}.
\end{equation}
Special cases include the Frobenius norm
$\|\ma\|_{(2)} = \|\ma\|_F$, and the two-norm $\|\ma\|_{(\infty)} = \|\ma\|_2$.
Here we consider only integer values of $p$. 
We want to find $k$ columns $\mc$ of $\ma\in\rmn$ 
with minimal Schatten $p$-norm, $\max_{\mc}{\|\mc\|_{(p)}}$.

%\ii{AKS: Not clear how to do it for $1 \le p < 2$ since the norm inequalities suggest $\|\mc\|_{(p)} \ge \|\mc\|_F = \sqrt{k}$. For $p = 2$, there is no inequality since $\|\mc\|_F = \sqrt{k}$. We cannot argue that $\|\mc\|_F = \sqrt{k}$ if and only if $\mc$ has orthonormal columns.}

\subsection{Decision problems}\label{s_twop1}
Given a matrix $\ma\in\rmn$ with unit-norm columns $\|\ma\ve_j\|_2=1$, $1\leq j\leq n$;
\changes{integer $1\leq k\leq \min\{\rank(\ma),m/3\}$};
and parameter $b>0$.
\begin{enumerate}
\item \emph{Two-norm minimization}:\\
Does $\ma$ have a submatrix $\mc\in\real^{m\times k}$ with 
$\|\mc\|_2 \leq b$?
\item \emph{Schatten $p$-norm  minimization}:\\
For $p>2$, does $\ma$ have a submatrix $\mc\in\real^{m\times k}$ with 
$\|\mc\|_{(p)} \leq b$?
\end{enumerate}

\begin{theorem}\label{t_norm}
    The decision problems: two-norm minimization and Schatten $p$-norm  minimization are NP hard.
\end{theorem}

\begin{proof}
    The proof is the same as that of Theorem~\ref{t_voldec}, except we use Lemma~\ref{l_norm} \changes{in Section~\ref{s_unitnorm3}} instead.
\end{proof}

Remark~\ref{r_schattenp} \changes{in Section~\ref{s_unitnorm3}} explains the requirement $p > 2$:
The meaningful bound for $p<2$ is an upper bound instead of a lower bound.

In the case $p=2$, minimizing the Frobenius norm 
can be accomplished by a polynomial time algorithm 
that relies on the identity 
\[ \|\mc\|_F^2 = \sum_{j=1}^k\|\mc\ve_j\|_2^2, \]
and constructs $\mc$ from $k$ columns of $\ma$ with smallest two-norm.

\subsection{No PTAS for two-norm minimization}\label{s_twop2}
Lemma~\ref{l_norm} \changes{in Section~\ref{s_unitnorm3}} 
implies that a matrix with unit-norm columns has a minimal two-norm equal to~1.
We give a rigorous derivation of the 
inapproximability threshold 
$2/\sqrt{3}$ from~\cite[Theorem 5]{CM09}.

\begin{theorem}\label{t_2m}
It is NP hard to approximate two-norm minimization within a factor of $2/\sqrt{3}$. 
Thus, unless P=NP, there is no PTAS for two-norm minimization. 
%There is no PTAS that can find,
%for all $\epsilon>0$, a submatrix $\mc_*\in\real^{m\times k}$ of~$\ma$ so that
%\begin{equation*}
%\|\mc_*\|_2\leq 2/\sqrt{3}-\epsilon
%\end{equation*}
%unless $P=NP$.
\end{theorem}

\begin{proof}
The setup is the same as in the proof of Theorem~\ref{t_volptas}, and we present an induction over
the number of columns $k$ in $\mc$ to show that $\|\mc\|_2\geq 2/\sqrt{3}$.

\begin{itemize}
\item For $k=2$, two distinct sets can overlap in one or
in two elements. 
When the sets share a single element, the nonzero elements (\ref{e_aux1}) of the relevant two columns
$\hat{\mc}_2$ have largest singular value 
$\|\hat{\mc}_2\|_2=\sigma_1(\hat{\mc}_2)=2/\sqrt{3}$.
When the sets share two elements, the nonzero elements
(\ref{e_aux2}) of the relevant two columns $\hat{\mc}_2$
have largest singular value $\|\hat{\mc}_2\|_2=\sigma_1(\hat{\mc}_2)=\sqrt{5/3}$. Between the two, (\ref{e_aux1}) represents the smaller of the two large singular values.
Hence, for matrices $\hat{\mc}_2$ with $k=2$ columns
$\|\hat{\mc}_2\|_2=\sigma_1(\hat{\mc}_2)\geq 2/\sqrt{3}$.

\item For matrices $\hat{\mc}_{k-1}$ with $k-1$ columns,
assume that 
\begin{equation*}
\|\hat{\mc}_{k-1}\|_2=\sigma_{1}(\hat{\mc}_{k-1})\geq 
2/\sqrt{3}.
\end{equation*}

\item Let $\hat{\mc}_k\equiv
\begin{bmatrix} \hat{\mc}_{k-1} & \vc\end{bmatrix}$
be a matrix with $k$ columns. From
singular value interlacing \cite[Corollary 8.6.3]{GovL13}
follows
$\sigma_1(\hat{\mc}_{k})\geq
\sigma_{1}(\hat{\mc}_{k-1})$. Together with the above induction hypothesis this implies
\begin{equation*}
\|\hat{\mc}_{k}\|_2=\sigma_{1}(\hat{\mc}_{k})\geq 
\sigma_1(\hat{\mc}_{k-1})\geq 2/\sqrt{3}.
\end{equation*}
\end{itemize}$\quad$
\end{proof}

\section{Pseudo-inverse norms}\label{s_psnorm}
We show that minimizing the Schatten $p$-norms of the pseudo-inverse is NP hard
(Section~\ref{s_psnorm1}), 
derive a new approximability threshold for minimizing the two-norm
of the pseudo-inverse (Section~\ref{s_psnorm2}), and show that 
the Frobenius norm (Section~\ref{s_psnorm3})
and Schatten $p$-norm (Section~\ref{s_psnorm4})
minimization of the pseudo-inverse do
not admit a PTAS.

We want to find $k$ columns $\mc$ of $\ma\in\rmn$ whose
pseudo-inverse has minimal Schatten $p$-norm, 
$\min_{\mc}{\|\mc^{\dagger}\|_{(p)}}$.

\subsection{Decision problems}\label{s_psnorm1}
Given a matrix $\ma\in\rmn$ with unit-norm columns $\|\ma\ve_j\|_2=1$, $1\leq j\leq n$;
\changes{integer $1\leq k\leq\min\{\rank(\ma),m/3\}$};
and parameter $b>0$.
\begin{enumerate}
\item Two-norm pseudo-inverse minimization:\\
Does $\ma$ have a submatrix $\mc\in\real^{m\times k}$ with 
$\|\mc^\dagger\|_2 \le b$?
\item Frobenius-norm pseudo-inverse minimization:\\
Does $\ma$ have a submatrix $\mc\in\real^{m\times k}$ with 
$\|\mc^\dagger\|_F \le b$?
\item Schatten $p$-norm pseudo-inverse minimization:\\
Does $\ma$ have a submatrix $\mc\in\real^{m\times k}$ with 
$\|\mc^\dagger\|_{(p)} \le b$?
\end{enumerate} 

\begin{theorem}\label{t_pi7}
    The decision problems: minimizing the pseudo-inverse norm for the two-norm, Frobenius norm, and Schatten $p$-norms for $p>2$ are NP hard.
\end{theorem}

\begin{proof}
    The proof for all the norms is the same as that of Theorem~\ref{t_voldec} but uses Lemma~\ref{l_pinv} instead. 
    
 %Additionally, Frobenius norm pseudo-inverse minimization is in NP since the pseudo-inverse $\mc^\dagger = (\mc^T \mc)^{-1}\mc^T$ and the Frobenius norm can be computed in polynomial time. Thus, Frobenius-norm pseudo-inverse minimization is NP complete.
\end{proof}

\subsection{No PTAS for two-norm minimization of the pseudo-inverse}\label{s_psnorm2}
Lemma~\ref{l_pinv} \changes{in Section~\ref{s_unitnorm5}} implies that the pseudo-inverse of a matrix 
with unit-norm columns has minimal two-norm of~1.
We derive an inapproximability threshold of~$\sqrt{3/2}$.

\begin{theorem}\label{t_ssv}
It is NP hard to approximate two-norm minimization of the pseudo-inverse within a factor of 
$2/\sqrt{3}$. Thus, unless P=NP, there is no PTAS for two-norm minimization of the pseudo-inverse.
%There is no PTAS that can find,
%for all $\epsilon>0$, a submatrix $\mc_*\in\real^{m\times k}$ of~$\ma$ so that
%\begin{equation*}
%\|\mc_*^{\dagger}\|_2\leq \sqrt{3/2}-\epsilon
%\end{equation*}
%unless $P=NP$.
\end{theorem}

\begin{proof}
The setup is the same as in the proof of Theorem~\ref{t_volptas}, and we present an induction over
the number of columns $k$ in $\mc$ to show that $\|\mc^{\dagger}\|_2\geq \sqrt{3/2}$.

\begin{itemize}
\item For $k=2$, two distinct sets can overlap in one or
in two elements. 
When the sets share a single element, the nonzero elements (\ref{e_aux1}) of the relevant two columns~$\hat{\mc}_2$ 
have smallest singular value 
$\sigma_2(\hat{\mc}_2)=\sqrt{2/3}$.
When the sets share two elements, the nonzero elements
(\ref{e_aux2}) of the relevant two columns~$\hat{\mc}_2$
have smallest singular value $\sigma_2(\hat{\mc}_2)=\sqrt{1/3}$. Between the two, (\ref{e_aux1}) represents the largest of the two small singular values.
Hence, for matrices $\hat{\mc}_2$ with two columns
$\|\hat{\mc}_2^{\dagger}\|_2=1/\sigma_2(\hat{\mc}_2)\geq\sqrt{3/2}$.

\item For matrices $\hat{\mc}_{k-1}$ with $k-1$ columns,
assume that 
\begin{equation*}
\|\hat{\mc}_{k-1}^{\dagger}\|_2=\sigma_{k-1}(\hat{\mc}_{k-1})\geq\sqrt{3/2}.
\end{equation*}

\item Let $\hat{\mc}_k\equiv
\begin{bmatrix} \hat{\mc}_{k-1} & \vc\end{bmatrix}$
be a matrix with $k$ columns. From
singular value interlacing \cite[Corollary 8.6.3]{GovL13}
follows
$\sigma_{k-1}(\hat{\mc}_{k-1})\geq
\sigma_{k}(\hat{\mc}_{k})$. Together with the above induction hypothesis this implies
\begin{equation*}
\|\hat{\mc}_{k}^{\dagger}\|_2=1/\sigma_{k}(\hat{\mc}_{k})\geq 1/\sigma_{k-1}(\hat{\mc}_{k-1})
\geq\sqrt{3/2}.
\end{equation*}
\end{itemize}$\quad$
\end{proof}

The inapproximability threshold of~$\sqrt{3/2}$
in Theorem~\ref{t_ssv}
differs from the one in~\cite[Theorem 7]{CM09},
which is $(\sqrt{3/2})^{\frac{1}{k-1}}$. The discrepancy
seems to be due to the estimation of the
smallest singular value from the volume.

\subsection{No PTAS for Frobenius norm minimization of the pseudo-inverse}\label{s_psnorm3}
Lemma~\ref{l_pinv} \changes{in Section~\ref{s_unitnorm4}}
implies that the pseudo-inverse of a matrix with $k$ unit-norm columns has a minimal Frobenius norm of~$\sqrt{k}$. 
We derive an inapproximability threshold of~$\sqrt{k+\tfrac{1}{4}}$.

\begin{theorem}\label{t_pi}
It is NP hard to approximate Frobenius norm minimization of the pseudo-inverse within a factor of 
$\sqrt{1+\tfrac{1}{4k}}$. Thus, unless P=NP, there is no PTAS for Frobenius norm minimization of the pseudo-inverse.
%There is no PTAS that can find, for all $\epsilon>0$,
%a submatrix $\mc_*\in\real^{m\times k}$ of $\ma$ so that 
%\begin{equation*}
%\|\mc_*^{\dagger}\|_F\leq\sqrt{k+\tfrac{1}{4}}-\epsilon
%\end{equation*}
%unless $P=NP$. 
\end{theorem}

\begin{proof}
The setup is the same as in the proof of Theorem~\ref{t_volptas}.
Assume that X3C instance is false, and assume that any two columns of~$\mc$ overlap.

We present induction over the number of columns $k$ to show that $\|\mc^\dagger\|_F \ge \sqrt{k+\tfrac14}$. 
\begin{itemize}
\item For $k=2$, two distinct sets can overlap in one or in two elements. 
When the sets share a single element, the nonzero elements (\ref{e_aux1}) of the relevant two columns
$\hat{\mc}_2$ have singular values 
$\sigma_1(\hat{\mc}_2)=2/\sqrt{3}$ and $\sigma_2(\hat{\mc}_2)=\sqrt{3/2}$.
When the sets share two elements, the nonzero elements
(\ref{e_aux2}) of the relevant two columns $\hat{\mc}_2$
have singular values $\sigma_1(\hat{\mc}_2)=\sqrt{5/3}$ and
$\sigma_2(\hat{\mc}_2)=\sqrt{1/3}$.
Between the two, (\ref{e_aux1}) achieves the smallest of two-norm of the pseudoinverse
Hence, for matrices $\hat{\mc}_2$ with $k=2$ columns, (\ref{e_aux1}) implies that
\begin{equation*}
\|\hat{\mc}_2\|_F^2=\left(\tfrac{1}{\sigma_1(\hat{\mc}_2)}\right)^2+
\left(\tfrac{1}{\sigma_2(\hat{\mc}_2)}\right)^2
=\left(\tfrac{\sqrt{3}}{2}\right)^2+\left(\sqrt{\tfrac{3}{2}}\right)^2=
2+ \tfrac{1}{4}=k+\tfrac{1}{4}.
\end{equation*}

\item For matrices $\hat{\mc}_{k-1}$ with $k-1$ columns, assume that 
\begin{equation*}
\|\hat{\mc}_{k-1}^{\dagger}\|_F^2\geq k-1+\tfrac{1}{4}.
\end{equation*}

\item Let $\hat{\mc}_k\equiv
\begin{bmatrix} \hat{\mc}_{k-1} & \vc\end{bmatrix}$
be a matrix with $k$ columns, where $\|\vc\|_2=1$. From Lemma~\ref{l_fi} \changes{in Section~\ref{s_partinv}}
follows
\begin{equation*}
\|\hat{\mc}_k^{\dagger}\|_F^2\geq \|\hat{\mc}_{k-1}^{\dagger}\|_F^2+\|\vc^{\dagger}\|_2^2
\geq \left(k-1 +\tfrac{1}{4}\right)+1=k+\tfrac{1}{4}.
\end{equation*}$\quad$
\end{itemize}
\end{proof}

\begin{comment}
The example in the proof of Theorem~\ref{t_ssv} illustrates
that if there is an overlap in an element, then this 
overlap must happen for at least three pairs of sets.
Thus, in the best case, $\mc$ has $k-6$ singular values equal to~1, 3 singular values equal to $2/\sqrt{3}$, and 3 singular
values equal to $\sqrt{2/3}$. Thus
\begin{equation*}
\|\mc^{\dagger}\|_F^2\geq k-6 + 3\, \frac{3}{4} + 3\,\frac{3}{2}=
k+\frac{3}{4}.
\end{equation*}

\ii{For the $m\times 2$ matrix $\hat{\mc}$ in (\ref{e_aux1}), 
$\|\hat{\mc}\|_F^2=2+\tfrac{1}{4}=k+\tfrac{1}{4}$.}
\end{comment}

\subsection{No PTAS for Schatten p-norm minimization of the pseudo-inverse}\label{s_psnorm4}
Lemma~\ref{l_pinv} \changes{in Section~\ref{s_unitnorm4}}
implies that the pseudo-inverse of a matrix with $k$ unit-norm columns has a minimal Schatten $p$-norm of~$k^{1/p}$. 
We derive an inapproximability threshold of $k^{1/p}\sqrt{(1+\tfrac{1}{4k})}$ for $p > 2$.

\begin{theorem}\label{t_pi3}
It is NP hard to approximate the Schatten $p$-norm minimization of the pseudo-inverse within a factor of 
$\sqrt{1+\tfrac{1}{4k}}$ for $p>2$. Thus, unless P=NP, there is no PTAS for Schatten $p$-norm minimization of the pseudo-inverse.
\end{theorem}
\begin{proof}
    The setup is the same as in the proof of Theorem~\ref{t_volptas}.
Assume that X3C instance is false, and assume that any two columns of~$\mc$ overlap. As in the proof of Theorem~\ref{t_pi} conclude that $\|\mc^\dagger\|_F \ge \sqrt{k + \tfrac{1}{4}}=k^{1/2}\sqrt{1+\tfrac{1}{4k}}$. For $p > 2 $, 
we interpret the Schatten $p$-norms as vector $p$-norms on the singular values, and apply the  relation~\cite[(5.4.21)]{HoJoI} between vector $p$-norms,
\[ \|\mc^\dagger\|_{(p)} \ge k^{(\tfrac{1}{p}-\tfrac{1}{2})} \|\mc^\dagger\|_F \ge k^{1/p} \sqrt{1 + \tfrac{1}{4k}}.\]
\end{proof}

\section{Condition numbers}\label{s_condno}
We show that the minimization of two-norm, Frobenius norm, Schatten $p$-norm condition numbers, mixed 
and Schatten $p$-norm mixed  condition numbers for $p>2$ is NP hard
(Section~\ref{s_condno1}), derive a new approximability threshold
for 2-norm condition minimization (Section~\ref{s_condno2}),
and show that Frobenius norm condition number minimization
 (Section~\ref{s_condno3}) and mixed condition number
 minimization (Section~\ref{s_condno4}) do not admit a PTAS.

For $p\geq 1$, the $p$-norm condition number with regard to left
inversion of a matrix $\mc\in\real^{m\times k}$ with $\rank(\mc)=k$ is defined as
\begin{equation*}
\kappa_{(p)}(\mc)\equiv \|\mc\|_{(p)}\|\mc^{\dagger}\|_{(p)}.
\end{equation*}
Special cases include the two-norm condition number
\begin{equation*}
\kappa_2(\mc) = \|\mc\|_2 \|\mc^\dagger\|_2=
\sigma_1(\mc)/\sigma_k(\mc),
\end{equation*}
and the Frobenius norm condition number
$\kappa_F(\mc) \equiv \|\mc\|_F \|\mc^\dagger\|_F$.

The mixed condition number \cite[Section 3]{Dem87} is
$\kappa_D(\mc)\equiv\|\mc\|_F\|\mc^{\dagger}\|_2$ 
and its extension to Schatten $p$-norms for $p> 2$ is
\begin{equation*}
\kappa_{D,p}(\mc) \equiv \|\mc\|_{(p)} \|\mc^\dagger\|_2.
\end{equation*}

We want to find $k$ columns~$\mc$ of $\ma\in\rmn$ with minimal
condition number, $\min_{\mc}{\kappa_{\xi}(\mc)}$
for $\xi\in\{2, F, (p), D, (D,p)\}$.

\subsection{Decision problems}\label{s_condno1}
Given a matrix $\ma\in\rmn$ with  unit-norm columns $\|\ma\ve_j\|_2=1$, $1\leq j\leq n$;
\changes{integer $1\leq k\leq \min\{\rank(\ma),m/3\}$};
and parameter $b>0$.
\begin{enumerate}
\item \emph{Two-norm condition number minimization:}\\
Does $\ma$ have a submatrix $\mc\in\real^{m\times k}$ with 
$\kappa_2(\mc)\leq b$?
\item \emph{Frobenius-norm condition number  minimization:} \\
Does $\ma$ have a submatrix $\mc\in\real^{m\times k}$ with 
$\kappa_F(\mc)\leq b$?
\item \emph{Schatten $p$-norm condition number  minimization: }\\
Does $\ma$ have a submatrix $\mc\in\real^{m\times k}$ with 
$\kappa_{(p)}(\mc)\leq b$? 
\item \emph{Mixed condition number  minimization:} \\
Does $\ma$ have a submatrix $\mc\in\real^{m\times k}$ with 
$\kappa_D(\mc)\leq b$?
\item \emph{Schatten $p$-norm mixed condition number  minimization} \\
Does $\ma$ have a submatrix $\mc\in\real^{m\times k}$ with 
$\kappa_{D,p}(\mc)\leq b$? 
\end{enumerate}

The NP hardness of two-norm condition number minimization was already established in~\cite[Theorem 4]{CM09}.

\begin{theorem}\label{t_cond}
    The decision problems of condition number minimization are NP hard: in the two-norm, Frobenius norm, Schatten $p$-norms for $p>2$, mixed norm, 
    and Schatten $p$-mixed norms for $p>2$.
%and NP complete for the Frobenius norm. 
\end{theorem}

\begin{proof}
The proof of NP hardness is the same as that of Theorem~\ref{t_voldec}, but uses Lemma~\ref{l_cond} instead. 
%The NP completeness of Frobenius norm condition number
%minimization follows from Theorem~\label{t_pi2}.
\end{proof}

\subsection{No PTAS for two-norm condition number minimization}\label{s_condno2}
Lemma~\ref{l_cond} \changes{in Section~\ref{s_unitnorm5}} implies that the minimal two-norm condition number is equal to~1. 
We derive an inapproximability threshold of $\sqrt{2}$.

\begin{theorem}\label{t_2cond}
It is NP hard to approximate two-norm condition number minimization within a factor of 
$\sqrt{2}$. Thus, unless P=NP, there is no PTAS for two-norm condition number minimization.
\end{theorem}

\begin{proof}
The setup is the same as in the proof of Theorem~\ref{t_volptas}.
We derive a lower bound for the condition numbers for the
reduced instance of X3C when it is false. Lemma~\ref{l_inter2}
\changes{in Section~\ref{s_unitnorm5}}
implies that the condition numbers decrease with the removal
of columns. Hence the smallest condition numbers are found
by identifying two columns for a false instance of X3C.

When the sets share a single element, the
nonzero elements (\ref{e_aux1}) of the relevant two columns 
$\hat{\mc}$ have
singular values $2/\sqrt{3}$ and $\sqrt{2/3}$.
When the sets share two elements, the
nonzero elements (\ref{e_aux2}) of the relevant two columns 
$\hat{\mc}$ have
singular values $\sqrt{5/3}$ and $1/\sqrt{3}$.
Between the two, (\ref{e_aux1}) achieves the smallest
condition number $(2/\sqrt{3})/(\sqrt{2/3})=\sqrt{2}$,
which implies $\kappa_2(\mc)\geq \sqrt{2}$.
\end{proof}

Theorem~\ref{t_2cond} derives a PTAS inapproximability threshold of $\sqrt{2}$
which differs from the one in \cite[Theorem 8]{CM09},
which is $(2^{2k-3}/3^{k-2})^{\frac{1}{2(k-1)}}$.

\subsection{No PTAS for Frobenius norm condition number minimization}\label{s_condno3}
Lemma~\ref{l_cond} \changes{in Section~\ref{s_unitnorm5}}
implies that a matrix with $k$ unit-norm columns
has  a minimal Frobenius norm condition
number of~$k$. 
We derive an inapproximability threshold of $\sqrt{k\left(k+\tfrac{1}{4}\right)}$.

\begin{theorem}\label{t_fcond}
It is NP hard to approximate Frobenius norm condition number minimization within a factor of 
$\sqrt{1+\tfrac{1}{4k}}$. Thus, unless P=NP, there is no PTAS for Frobenius norm condition number minimization.
\end{theorem}

\begin{proof}
From Theorem~\ref{t_pi} follows 
$\|\mc^{\dagger}\|_F\geq \sqrt{k+\tfrac{1}{4}}$.
Regardless of overlaps, the columns of $\mc$ have unit norm, so (\ref{e_srk})
\changes{in Section~\ref{s_unitnorm1}} implies
$\|\mc\|_F=\sqrt{k}$. Thus 
\begin{equation*}
\kappa_F(\mc)\geq \sqrt{k(k+\tfrac{1}{4})}=k\sqrt{1+\tfrac{1}{4k}}.
\end{equation*}
\end{proof}

\subsection{No PTAS for mixed condition number minimization}\label{s_condno4}
Lemma~\ref{l_cond} \changes{in Section~\ref{s_unitnorm5}}
implies that a matrix with $k$ unit-norm columns has a mixed condition
number with a minimal value of $\sqrt{k}$.
We  derive an approximability threshold of~$\sqrt{\tfrac{3}{2}k}$.

\begin{theorem}\label{t_cd}
It is NP hard to approximate mixed condition number minimization within a factor of 
$\sqrt{3/2}$. Thus, unless P=NP, there is no PTAS for mixed condition number minimization.
\end{theorem}

\begin{proof}
The setup is the same as in the proof of Theorem~\ref{t_volptas}.
Assume that X3C instance is false, and that any two columns of~$\mc$ overlap by at least one element.
Theorem~\ref{t_ssv} implies
\begin{equation*}
\kappa_D(\mc)^2=\|\mc\|_F^2\|\mc^{\dagger}\|_2^2\geq
\frac{3}{2}\,\|\mc\|_F^2.
\end{equation*}
Regardless of overlaps, the columns of $\mc$ have unit norm, so (\ref{e_srk})
\changes{in Section~\ref{s_unitnorm1}} implies
$\|\mc\|_F^2=k$. Thus $\kappa_D(\mc)^2\geq \frac{3}{2}k$.
\end{proof}

\iffalse
\subsection{No PTAS for Schatten \texorpdfstring{$p$}{p} norm condition number minimization}\label{s_condno5}
Lemma~\ref{l_cond}
implies that a matrix with $k$ unit-norm columns has a mixed condition
number with a minimal value of ${k}^{2/p}$.
We  derive an inapproximability threshold of $k^{1/p} \zeta_p \sqrt{1 + \tfrac{1}{4k}}$, where $\zeta_p$ is defined in Theorem~\ref{t_pm}.

\begin{theorem}\label{t_noptas_schatten}
It is NP hard to approximate mixed condition number minimization within a factor of 
$\zeta_p \sqrt{1 + \tfrac{1}{4k}}$. Thus, unless P=NP, there is no PTAS for -norm condition number minimization.
\end{theorem}

\begin{proof}
The setup is the same as in the proof of Theorem~\ref{t_volptas}.
Assume that X3C instance is false, and assume that any two columns of~$\mc$ overlap by at most
one element, and the overlapping pair is orthogonal to all other columns.

From the proof of Theorem~\ref{t_pi3}, we have 
\[ \|\mc^\dagger\|_{(p)} \ge k^{1/p}\sqrt{1 + \tfrac{1}{4k}},\]
whereas from the proof of Theorem~\ref{t_pm}, we get $\|\mc\|_{(p)} \ge \zeta_p$. Combine the two to get 
\[ \kappa_{(p)} \ge k^{1/p} \zeta_p \sqrt{1 + \tfrac{1}{4k}}. \]
\end{proof}
\fi

\section{Stable rank}\label{s_srank}
We show that maximization of the stable rank and the $p$-stable rank for $p>2$ is NP hard 
(Section~\ref{s_srank1}),
and show that stable rank maximization has no PTAS (Section~\ref{s_srank2}).

The stable rank of a nonzero matrix $\mc \in \real^{m\times k}$ is defined as 
\[ \sr(\mc) \equiv {\|\mc\|_F^2}/{\|\mc\|_2^2},\]
and $\sr(\vzero_{m\times k})=0$.
The extension to Schatten $p-$norms for $p> 2$
\cite{IpsS24} is
\[ \sr_{(p)}(\mc) \equiv {\|\mc\|_{(p)}^p}/{\|\mc\|_2^p},\] 
and  $\sr_{(p)}(\vzero_{m\times k})=0$.
Obviously, $\sr_{(2)}(\ma) = \sr(\ma)$.

We want to find $k$ columns $\mc$ of $\ma\in\rmn$ that
maximize the $p$-stable rank $\max_{\mc}{\sr_{(p)}(\mc)}$ for $p\geq 2$.

\subsection{Decision problems}\label{s_srank1}
Given a matrix $\ma\in\rmn$ with unit-norm columns $\|\ma\ve_j\|_2=1$, $1\leq j\leq n$;
\changes{$1\leq k\leq \min\{\rank(\ma), m/3\}$;}
and parameter $b>0$.
\begin{enumerate}
\item \emph{Stable rank maximization}:\\
Does $\ma$ have a submatrix $\mc\in\real^{m\times k}$ with $\sr(\mc)\geq b$?
\item \emph{$p$-stable rank maximization}:\\
Does $\ma$ have a submatrix $\mc\in\real^{m\times k}$ with $\sr_p(\mc)\geq b$?
\end{enumerate}

\begin{theorem}\label{t_other}
The  decision problems of stable rank maximization and $p$-stable rank maximization for $p>2$ are NP hard.
\end{theorem}

\begin{proof}
The proof is the same as that of Theorem~\ref{t_voldec} but uses Lemma~\ref{l_srank} 
\changes{in Section~\ref{s_unitnorm6}} instead. 
 \end{proof}

\subsection{No PTAS for stable rank maximization}\label{s_srank2}
Lemma~\ref{l_srank} \changes{in Section~\ref{s_unitnorm6}}
implies that, for a matrix with $k$ unit-norm columns, the 
maximal value of the stable rank equals $k$.
We derive an inapproximability threshold of~$\tfrac{3}{4}k$.

\begin{theorem}\label{t_sr}
It is NP hard to approximate stable rank 
maximization within a factor of $3/4$. Thus, unless P=NP, there is no PTAS for stable rank maximization.
%There is no PTAS that can find, for all $\epsilon>0$,
%a submatrix $\mc_*\in\real^{m\times k}$ of $\ma$ so that 
%\begin{equation*}
%\sr(\mc_*) \geq \tfrac{3}{4}k+\epsilon 
%\end{equation*}
%unless $P=NP$.
\end{theorem}

\begin{proof}
The setup is the same as in the proof of Theorem~\ref{t_volptas}.
Assume that X3C instance is false, in which 
case Theorem~\ref{t_2m} implies 
$\|\mc\|_2\geq 2/\sqrt{3}$.
Hence
\begin{equation*}
\sr(\mc)=\frac{\|\mc\|_F^2}{\|\mc\|_2^2}\leq\frac{3}{4}\|\mc\|_F^2=
\frac{3}{4}k,
\end{equation*}
because,
regardless of overlaps, the columns of $\mc$ have unit norm, 
so (\ref{e_srk}) \changes{in Section~\ref{s_unitnorm1}} implies $\|\mc\|_F^2=k$. 
\end{proof}

%\ii{PTAS for $p$-stable rank?}
%\aks{We can show for $p > 2$ $$\sr_{(p)}(\mc) = \frac{\|\mc\|_{(p)}^p}{\|\mc\|_2^p } \le k^{p/2}  \left( \frac{\sqrt{3}}{2}\right)^p. $$
%} 
%\ii{We could use the same approach for $p$-norm condition number estimation. However, it is legitimate if we don't use a PTAS for p-norm approximation? This may be unrealistically optimistic.}
\section{Optimal values for subset selection criteria}\label{s_unitnorm}
To express the optimization versions of the subset selection criteria as decision problems, we derive  optimal values for: maximal volume and S-optimality (Section~\ref{s_unitnorm1}),
maximal relative volume (Section~\ref{s_unitnorm2}), minimal two-norm and Schatten $p$-norms
(Section~\ref{s_unitnorm3}), minimal pseudo-inverse norms (Section~{\ref{s_unitnorm4}),
minimal condition numbers (Section~\ref{s_unitnorm5}), and maximal stable ranks
(Section~\ref{s_unitnorm6}).

\subsection{Maximal volume and S-optimality}\label{s_unitnorm1}
We show that for matrices with unit-norm columns, the
volume and S-optimality are bounded above by~1; and that
only matrices with orthonormal columns have
a maximal volume and S-optimality equal to~1.

\begin{lemma}\label{l_vol}
Let $\mc \in \mathbb{R}^{m \times k}$ with $\rank(\mc) = k$ and $\|\mc\ve_j\|_2=1$, $1\leq j\leq k$. Then each of the following inequalities holds with equality if and only if $\mc$ has orthonormal
columns. 
\begin{enumerate}
    \item Volume: $\vol(\mc)\leq 1$.
    \item S-optimality: $\sopt(\mc)\leq 1$.
\end{enumerate}
    \end{lemma}

\begin{proof} 
The assumptions imply 
\begin{align}\label{e_srk}
\|\mc\|_F^2 = \sum_{j=1}^k \|\mc\ve_j\|_2^2 = k.
\end{align}
%We also make use of the singular value ordering 
%$\sigma_1(\mc)\geq \cdots\geq\sigma_k(\mc)>0$, and
We tackle each criterion, in turn.
    \begin{enumerate}
      \item \emph{Volume maximization:}
From~(\ref{e_srk}) and the relation between the arithmetic and geometric means follows
\begin{align}\label{e_mean}
\begin{split}
1=\frac{\|\mc\|_F^2}{k}=\frac{\sigma_1(\mc)^2+\cdots +\sigma_k(\mc)^2}{k}&\geq \sqrt[k]{\sigma_1(\mc)^2\cdots \sigma_k(\mc)^2}\\
&=\vol(\mc)^{2/k}.
\end{split}
\end{align}
Thus $\vol(\mc)\leq 1$, which shows the inequality. Alternatively this also follows from Hadamard's inequality 
\cite[Corollary 7.8.3]{HoJoI}
and the fact that the columns have unit two-norm.

As for the equality, if $\mc$ has orthonormal columns, then all singular values $\sigma_j(\mc)=1$, $1\leq j\leq k$. Hence $\vol(\mc)=1$.

Conversely, if $\vol(\mc)=1$, then $\vol(\mc)^{2/k}=1$, and
(\ref{e_mean}) implies that  the arithmetic and geometric means are the same. This can only happen if all $\sigma_j(\mc)$ are equal. Thus, $\sigma_j(\mc)=1$, $1\leq j\leq k$, and $\mc$ has orthonormal columns.
        \item  \emph{S-optimality maximization:} The inequality follows from Hadamard's inequality \cite[Theorem 7.8.1]{HoJoI},
\begin{equation*}
\vol(\mc)^2=\det(\mc^T\mc)\leq 
\prod_{i=1}^k{(\mc^T\mc)_{ii}}=\prod_{i=1}^k{\|\mc\ve_i\|_2^2}.
\end{equation*}
The equality follows from the assumption that $\mc$
has unit-norm columns, so that 
$\sopt(\ms)=\vol(\mc)^{1/k}$, and from the previous part.
A more complicated proof of the equality 
is presented in \cite[Theorem 3.2]{SX2016}.
    \end{enumerate}
\end{proof}

\subsection{Maximal relative volume}\label{s_unitnorm2}
We show that the relative volume is bounded above by~1, and that only matrices with orthonormal columns have a relative volume equal to~1.

\begin{lemma}\label{lem:orth} 
Let $\mc\in \mathbb{R}^{m \times k}$ with $\rank(\mc)=k$. Then $0<\rvol(\mc) \le 1$.

Furthermore, $\rvol(\mc) = 1$  if and only if the columns of $\mc$ are orthonormal.
\end{lemma}
\begin{proof}
The inequality follows from the singular
value ordering~(\ref{e_so}).

As for the equality, if $\mc$ has orthonormal columns, then all singular values of $\mc$ are
equal to~1,
and $\rvol(\mc) = 1$. Conversely, if $\rvol(\mc) = 1$ then 
\begin{equation}\label{e_ortho}
\prod_{j=1}^k{\frac{\sigma_j(\mc)}{\sigma_1(\mc)}} = 1.
\end{equation}
Since all factors $\sigma_j(\mc)/\sigma_1(\mc)\leq 1$, (\ref{e_ortho}) can only 
hold if all factors $\sigma_j(\mc)/\sigma_1(\mc)=1$.
Hence, all singular values of $\mc$ are equal to~1, and $\mc$ has orthonormal columns.
 \end{proof}

We show that removal of columns can only
 increase the relative volume. 
 
\begin{lemma}\label{l_inter}
Let $\mc\in\real^{m\times k}$ with $\rank(\mc)=k$,
and let $\mc_{\ell}\in\real^{m\times \ell}$ be a 
column submatrix of $\mc$, $1\leq \ell<k$.
Then
\begin{equation*}
\rvol(\mc)\leq \rvol(\mc_{\ell})\leq \rvol(\mc_1)=1.
\end{equation*}
\end{lemma}

\begin{proof}
Since $\mc$ has full column rank, all columns are non-zero,
and any non-zero column $\mc_1$ attains the largest relative volume of~1.

For $\ell\geq 2$, singular value interlacing (\ref{e_inter}) implies
\begin{equation*}
\sigma_{k-\ell+j}(\mc)\leq\sigma_j(\mc_{\ell})\leq \sigma_j(\mc),\qquad 1\leq j\leq \ell.
\end{equation*}
Since all factors $\sigma_j(\mc)/\sigma_1(\mc)\leq 1$,
\begin{align*}
\rvol(\mc)=\prod_{j=1}^k{\frac{\sigma_j(\mc)}{\sigma_1(\mc)}}\leq 
\prod_{j=1}^{\ell}{\frac{\sigma_{k-\ell+j}(\mc)}{\sigma_1(\mc)}}\leq 
\prod_{j=1}^{\ell}{\frac{\sigma_j(\mc_{\ell})}{\sigma_1(\mc_{\ell})}} 
=\rvol(\mc_{\ell}).
\end{align*}
\end{proof}

\subsection{Minimal two-norm and Schatten p-norms}\label{s_unitnorm3}
We show that only matrices with $k$ orthonormal columns have a
two-norm equal to~1, and a Schatten $p$-norm 
equal to~$k^{1/p}$ for $p>2$.
In the special case $p=2$, the equality
 $\|\mc\|_{(2)}=\|\mc\|_F=\sqrt{k}$ is always true for unit-norm columns, regardless of whether they are orthonormal.

\begin{lemma}\label{l_norm}
Let $\mc \in \mathbb{R}^{m \times k}$ with $\rank(\mc) = k$ and $\|\mc\ve_j\|_2=1$, $1\leq j\leq k$. Then each of the following inequalities holds with equality if and only if $\mc$ has orthonormal
columns.
\begin{itemize}
\item \emph{Two-norm}: $\|\mc\|_2 \ge 1$
\item \emph{Schatten $p$-norm}: $\|\mc \|_{(p)} \ge k^{1/p}$ for $p > 2$.
\end{itemize}
\end{lemma}
\begin{proof}  We tackle the two cases separately. 
\begin{enumerate}
    \item \emph{Two-norm}: The inequality follows from the definition of the norm 
    \[ \|\mc\|_2 \ge \max_{1 \le i \le k} \|\mc \ve_i\|_2 = 1.  \] 
    As for the equality, if $\mc$ has orthonormal columns, then $\|\mc\|_2 = \sigma_1(\mc) = 1$. 
    
    Conversely, if $\|\mc\|_2= 1$, then~\eqref{e_srk} \changes{in Section~\ref{s_unitnorm1}}
    and the singular value ordering imply $\sigma_1(\mc) = \dots = \sigma_k(\mc) = 1$. That is, $\mc$ has orthonormal columns. 
    \item \emph{Schatten $p$-norm for $p>2$}: The norm inequalities~\cite[Equation (5.4.21)]{HoJoI} imply
    \[ \sqrt{k}  = \|\mc\|_F \le k^{(\tfrac{1}{2}- \tfrac{1}{p})} \|\mc\|_{(p)}, \qquad p >2.  \] 
    Rearranging gives the inequality $\|\mc \|_{(p)} \ge k^{1/p}$ for $p> 2$.
   % Note that $\|\mc\|_{(2)}=\|\mc\|_F=\sqrt{k}$.
    
    %Similarly, for $1 \le p < 2 $, $\|\mc\|_{(p)} \ge \|\mc\|_F= \sqrt{k} \ge k^{1/p}$. 

    As for the equality, we use the same idea as in the proof of~\cite[Proposition 3.1]{eswar2024bayesian}. If $\mc$ has orthonormal columns, then all the singular values of $\mc$ are equal to~$1$ and $\|\mc\|_{(p)} = k^{1/p}$. 
    
    Conversely, if $\|\mc\|_{(p)} = k^{1/p}$, then (\ref{e_srk}) \changes{in Section~\ref{s_unitnorm1}} implies
    \begin{equation}\label{e_83}
    \sum_{j=1}^k \left(  \frac{2}{p}\sigma_j(\mc)^p  - \sigma_j(\mc)^2 + (1-\tfrac{2}{p}) \right)  = 0.
    \end{equation}
    This sum has the form 
    \begin{equation*}
    \sum_{j=1}^kg_p(\sigma_j(\mc)) = 0\qquad \text{where} \quad g_p(x) \equiv \tfrac{2x^p}{p} - x^2 + (1-\tfrac{2}{p}).
    \end{equation*}
    For $x >  0$, $g_p(x)$ has the unique global minimizer $x_*=1$ with $g_p(x_*) = 0$. 
    Since the singular values are positive, $\sum_{j=1}^kg(\sigma_j(\mc))$ is a sum of nonnegative summands. Thus equality in (\ref{e_83}) is achieved only if $\sigma_1(\mc) = \dots = \sigma_k(\mc) = 1$, implying that $\mc$ has orthonormal columns.  
    \end{enumerate}   
    \end{proof}

Below is a justification for  limiting Schatten $p$-norms
to the case $p\geq 2$: The lower bound $\|\mc\|_{(p)}\geq k^{1/p}$ for $p>2$ turns into an upper bound for $1\leq p<2$. 

\begin{remark}\label{r_schattenp}
Let $\mc \in \mathbb{R}^{m \times k}$ with $\rank(\mc) = k$ and $\|\mc\ve_j\|_2=1$, $1\leq j\leq k$.
For $1 \le p < 2$, \cite[(5.4.21)]{HoJoI} implies
    \[ \sqrt{k} \le \|\mc\|_{(p)} \le k^{(\tfrac{1}{p}-\tfrac12)} \sqrt{k} = k^{1/p} \]
and $\|\mc\|_{(p)} = k^{1/p}$ if and only if $\mc$ has orthonormal columns.
\end{remark}

\subsection{Minimal pseudo-inverse norms}\label{s_unitnorm4}
We show that matrices with $k$ unit-norm columns have a
pseudo-inverse whose two-norm is bounded below by~$1$, and
whose $p$-norm is bounded below by $k^{1/p}$ for $p\geq 2$; and that
only matrices with $k$ orthonormal columns have pseudo-inverses with
a two-norm equal to~1, and a Schatten $p$-norm equal
to~$k^{1/p}$ for $p>2$.

\begin{lemma}\label{l_pinv}
Let $\mc \in \mathbb{R}^{m \times k}$ with $\rank(\mc) = k$ and $\|\mc\ve_j\|_2=1$, $1\leq j\leq k$. Then each of the following inequalities holds with equality if and only if $\mc$ has orthonormal
columns.
    \begin{itemize}
\item \emph{Two-norm of pseudo-inverse}: $\|\mc^\dagger \|_2\geq 1$
\item \emph{Frobenius-norm of pseudo-inverse}: $\|\mc^\dagger \|_F\geq \sqrt{k}$
\item \emph{Schatten $p$-norm of pseudo-inverse}: $\|\mc^\dagger \|_{(p)}\geq k^{1/p}$.
\end{itemize}

\end{lemma}
\begin{proof}
We tackle each case separately. 

\begin{itemize}
\item \emph{Two-norm of the pseudo-inverse:}
From~(\ref{e_srk}) \changes{in Section~\ref{s_unitnorm1}} follows
\begin{equation*}
k=\|\mc\|_F^2=\sum_{j=1}^k{\sigma_j(\mc)^2}
\geq k\sigma_k(\mc)^2
\end{equation*}
Thus $\sigma_k(\mc)\leq 1$ and 
$\|\mc^{\dagger}\|_2=1/\sigma_k(\mc)\geq 1$, which shows the inequality.

As for the equality, if $\mc$ has orthonormal columns, then
$\mc^{\dagger}=\mc^T$, hence $\|\mc^{\dagger}\|_2=1$.

Conversely, if $\|\mc^{\dagger}\|_2 = 1$, then the singular value ordering $\sigma_1(\mc) \ge \cdots \ge \sigma_k(\mc) = 1$ and (\ref{e_srk}) \changes{in Section~\ref{s_unitnorm1}}
imply $\sigma_j(\mc)= 1$,  $1\leq j\leq k$. Thus, $\mc$ has orthonormal columns.

\item \emph{Frobenius norm of the pseudo-inverse:} This is a special case of the Schatten $p$-norm for $p=2$.
\item \emph{Schatten p-norm of the pseudo-inverse:} From
\[  k = \|\mc\|_F^2 = \sum_{j=1}^k \sigma_j(\mc)^2 \le k \sigma_1(\mc)^2\]
follows $\sigma_1(\mc) \ge 1$. Therefore,  $$\|\mc^\dagger\|_{(p)} \ge \left(\frac{k}{\sigma_1(\mc)^p} \right)^{1/p}  \ge k^{1/p}. $$
As for the equality, if $\mc$ has orthornormal columns, then $\sigma_j(\mc) = 1$, $1\leq j\leq k$. Hence, $\|\mc^\dagger\|_{(p)} = k^{1/p}$. 

Conversely, if $\|\mc^\dagger\|_{(p)} = k^{1/p}$, then~\eqref{e_srk} \changes{in Section~\ref{s_unitnorm1}} implies 
\begin{equation}\label{e_84}
0 = \|\mc\|_F^2 +\frac{2}{p}\|\mc^\dagger\|_{(p)}^p - k(1+\tfrac{2}{p}) = \sum_{j=1}^k (\sigma_j(\mc)^2 + \tfrac{2}{p\sigma_j(\mc)^p} - (1+\tfrac{2}{p})).   
\end{equation}
Each summand equals $g(\sigma_j(\mc))$ for $g(x)\equiv x^2 + \tfrac{2}{px^p} - (1+\tfrac{2}{p})$.  For $x > 0$, the function $g(x)$ is nonnegative with  global minimizer $x_* = 1$ and $g(x_*) = 0$. Thus, equality in
(\ref{e_84}) is achieved if and only if $\sigma_j(\mc)=1$, $1\leq j\leq k$,
and $\mc$ has orthonormal columns. 
\end{itemize}
\end{proof}

\subsection{Minimal condition numbers}\label{s_unitnorm5}
We derive lower bounds for condition numbers
of matrices with unit-norm columns, and show that 
the lower bounds are attained only by matrices 
with orthonormal columns.

\begin{lemma}\label{l_cond}
Let $\mc \in \mathbb{R}^{m \times k}$ with $\rank(\mc) = k$ and $\|\mc\ve_j\|_2=1$, $1\leq j\leq k$. Then each of the following inequalities holds with equality if and only if $\mc$ has orthonormal
columns.
\begin{itemize}
\item Two-norm condition number: $\kappa_2(\mc)\geq 1$
 \item Frobenius norm condition number: $\kappa_F(\mc) \ge k$
\item Schatten $p$-norm condition number: $\kappa_{(p)}(\mc) \ge k^{2/p}$
\item Mixed condition number: $\kappa_D(\mc)   \ge \sqrt{k}$
\item Schatten $p$-norm mixed condition number $\kappa_{D,p}(\mc)  \ge k^{1/p}$ 
\end{itemize}
\end{lemma}

\begin{proof}
We tackle each case separately.
\begin{enumerate}
\item \emph{Two-norm condition number:} The inequality follows from the singular value ordering~(\ref{e_so}).

As for the equality, if $\mc$ has orthonormal columns, then all its singular values are $1$ and $\kappa_2(\mc) = 1$. Conversely, if $\kappa_2(\mc) = \sigma_1(\mc)/\sigma_k(\mc) = 1$, then the singular value ordering~(\ref{e_so}) ensures
that $\sigma_j(\mc)=1$, $1\leq j\leq k$, and $\mc$ has orthonormal columns.  
 \item \emph{Frobenius norm condition number:} 
The inequality follows from applying the Cauchy-Schwartz inequality to the vectors 
 \begin{equation*}
 \vx \equiv \begin{bmatrix} \sigma_1(\mc) & \cdots & \sigma_k(\mc)
        \end{bmatrix}^T, \qquad
\vy \equiv  \begin{bmatrix} 1/\sigma_1(\mc) & 
\cdots & 1/\sigma_k(\mc) \end{bmatrix}^T,
\end{equation*}
which shows
\begin{equation*}
k = \vx^T \vy \le \|\vx\|_2 \|\vy\|_2 = \left(\sum_{j=1}^k\sigma_j(\mc)^2 \right)^{1/2} \left(\sum_{j=1}^k\sigma_j(\mc)^{-2}\right) ^{1/2}= \kappa_F(\mc).
\end{equation*}
As for the equality, if $\mc$ has orthonormal columns, then $\kappa_F(\mc) = k$. 

Conversely, if $\kappa_F(\mc) = k$, then  (\ref{e_srk}) \changes{in Section~\ref{s_unitnorm1}}
implies $\|\mc\|_F=\sqrt{k}$, hence
$\|\mc^{\dagger}\|_F^2=\sum_{j=1}^k\sigma_j(\mc)^{-2} = k$. 
As in the proof of Lemma~\ref{l_pinv}, one shows $\sigma_j(\mc) = 1$, $1\leq j\leq k$. Thus, $\mc$ has orthonormal columns.

\item \emph{Schatten $p$-norm condition number for $p>2$}: The inequality follows from applying the Cauchy-Schwartz inequality to the vectors 
 \begin{equation*}
 \vx \equiv \begin{bmatrix} \sigma_1(\mc)^p & \cdots & \sigma_k(\mc)^p
        \end{bmatrix}^T, \qquad
\vy \equiv  \begin{bmatrix} 1/\sigma_1(\mc)^p & 
\cdots & 1/\sigma_k(\mc)^p \end{bmatrix}^T,
\end{equation*}
which shows
\begin{equation*}
k = \vx^T \vy \le \|\vx\|_2 \|\vy\|_2 = \left(\sum_{j=1}^k\sigma_j(\mc)^p \right)^{1/2} \left(\sum_{j=1}^k\sigma_j(\mc)^{-p}\right) ^{1/2}= \kappa_p(\mc)^{p/2}.
\end{equation*}
Equality occurs in the Cauchy-Schwartz inequality only if $\vx$ is a multiple of~$\vy$. This ensures that all singular values are equal to~$1$, thus $\mc$ has orthonormal columns.

Conversely, if $\mc$ has orthonormal columns, then all the singular values are equal to $1$ and $\kappa_{(p)}(\mc) = k^{2/p}$. 

\item \emph{Mixed condition number}: 
The inequality follows from the definition of $\kappa_D(\mc)$. As for the equality, if $\mc$ has orthonormal columns, then $\kappa_D(\mc)=\sqrt{k}$.
Conversely, if $\kappa_D(\mc) = \sqrt{k}$, then (\ref{e_srk}) \changes{in Section~\ref{s_unitnorm1}}
implies $\sigma_k(\mc) = 1$. 
The singular value ordering $\sigma_1(\mc) \ge \cdots \ge \sigma_k(\mc) = 1$ and (\ref{e_srk}) imply $\sigma_j(\mc)= 1$,  $1\leq j\leq k$. Thus $\mc$ has orthonormal columns.

\item \emph{Schatten $p$-norm mixed condition number for $p>2$}: 
The inequality follows from the definition of $\kappa_{D,p}(\mc)$. 
As for the equality, if $\mc$ has orthonormal columns, then $\kappa_{D,p}(\mc)={k}^{1/p}$.

Conversely, if $\kappa_{D,p}(\mc) = {k}^{1/p}$, then $\sum_{j=1}^{\changes{k}}\sigma_j(\mc)^p = k\sigma_k(\mc)^p$,
so that 
\begin{equation*}
\sum_{j=1}^{\changes{k}} ( \sigma_j(\mc)^p -\sigma_k(\mc)^p) = 0.
\end{equation*}
The singular value ordering implies that all singular values are equal to~$1$, thus  $\mc$ has orthonormal columns.
\end{enumerate}
\end{proof}

We show that the removal of columns can only
 increase the condition numbers. 

\begin{lemma}\label{l_inter2}
Let $\mc\in\real^{m\times k}$ with $\rank(\mc)=k$,
and let $\mc_{\ell}\in\real^{m\times \ell}$ 
be a column submatrix of $\mc$, $1\leq \ell<k$.
Then
\begin{align*}
\kappa_{\xi}(\mc)&\geq \kappa_{\xi}(\mc_{\ell})\geq \kappa_{\xi}(\mc_1)=1, \qquad \xi\in\{2, F, (p), D, (D,p)\}.
\end{align*}
\end{lemma}

\begin{proof}
Since $\mc$ has full column rank, all columns are non-zero,
and any non-zero column $\mc_1$ attains the minimal 
value of~1 for all the condition numbers.
\begin{itemize}
\item \emph{Two-norm condition number:}
For $\ell\geq 2$, singular value interlacing (\ref{e_inter}) implies
$\sigma_1(\mc)\geq\sigma_1(\mc_{\ell})$ and 
$\sigma_k(\mc)\leq \sigma_{\ell}(\mc_{\ell})$, 
hence 
\begin{equation*}
\kappa_2(\mc)=\frac{\sigma_1(\mc)}{\sigma_{k}(\mc)}\geq
\frac{\sigma_1(\mc_{\ell})}{\sigma_{\ell}(\mc_{\ell})}=\kappa_2(\mc_{\ell}).
\end{equation*}
\item \emph{Frobenius norm and Schatten-$p$ norm condition numbers for $p>2$:}
We give the proof for $p\geq 2$, which includes the case of the Frobenius norm. For $\ell\geq 2$, singular value interlacing~(\ref{e_inter})
implies
\begin{equation*}
\sigma_{k-\ell+j}(\mc)\leq\sigma_j(\mc_{\ell})\leq \sigma_j(\mc),\qquad 1\leq j\leq \ell.
\end{equation*}
Thus
\begin{equation*}
\|\mc\|_{(p)}^p=\sum_{j=1}^k{\sigma_j(\mc)^p}\geq
\sum_{j=1}^{\ell}{\sigma_j(\mc)^p}\geq
\sum_{j=1}^{\ell}{\sigma_j(\mc_{\ell})^p}
=\|\mc_{\ell}\|_{(p)}^p
\end{equation*}
and 
\begin{equation*}
\|\mc^{\dagger}\|_{(p)}^p=\sum_{j=1}^k{\frac{1}{\sigma_j(\mc)^p}}\geq
\sum_{j=k-\ell+1}^{\ell}{\frac{1}{\sigma_j(\mc)^p}}\geq
\sum_{j=1}^{\ell}{\frac{1}{\sigma_j(\mc_{\ell})^p}}
=\|\mc_{\ell}^{\dagger}\|_{(p)}^p.
\end{equation*}
Thus
\begin{equation*}
\left(\kappa_{(p)}(\mc)\right)^p=\|\mc\|_{(p)}^p\|\mc^{\dagger}\|_{(p)}^p
\geq \|\mc_{\ell}\|_{(p)}^p\|\mc_{\ell}^{\dagger}\|_{(p)}^p
=\left((\kappa_{(p)}(\mc_{\ell})\right)^p.
\end{equation*}
\item \emph{Mixed condition number and Schatten $p$-norm mixed condition numbers:} The proofs are analogous to the ones above.
\end{itemize}
\end{proof}

\subsection{Maximal stable rank}\label{s_unitnorm6}
We derive upper bounds for the stable rank
and $p$-stable rank for $p>2$
of matrices with unit-norm columns, and show that 
the upper bounds are attained only by matrices 
with orthonormal columns.

\begin{lemma}\label{l_srank}
Let $\mc \in \mathbb{R}^{m \times k}$ with $\rank(\mc) = k$ and $\|\mc\ve_j\|_2=1$, $1\leq j\leq k$. Then each of the following inequalities holds with equality if and only if $\mc$ has orthonormal
columns.
    \begin{enumerate}
\item Stable rank: $\sr(\mc) \le k$
%  \item Stable rank of the pseudo-inverse: \sr(\mc^{\dagger}) \le k$
\item $p$-stable rank: $\sr_{(p)}(\mc) \le k$
%\item $p$-stable rank of pseudo-inverse: $\sr_{(p)}(\mc^\dagger) \le k$.
\end{enumerate}
\end{lemma}

\begin{proof}
We tackle each part in turn.
    \begin{enumerate}
\item \emph{Stable rank:} The inequality follows from the definition.  

As for the equality, if $\mc$ has orthonormal columns, then
all $k$ singular values equal to $1$, and $\sr(\mc) = k$. 
Conversely, if $k=\sr(\mc) = \|\mc\|_F^2/\|\mc\|_2^p$, then the definition of stable rank 
implies 
\begin{equation*}
\sum_{j=1}^k {\sigma_j(\mc)^2}=\|\mc\|_F^2
=k\|\mc\|_2^2=k\,\sigma_1(\mc)^2.
\end{equation*}
Hence $\sum_{j=1}^k{(\sigma_1(\mc)^2-\sigma_j(\mc)^2)}=0$.
From the singular value ordering (\ref{e_so}) follows that this
is a sum of non-negative summands, and every
summand must be equal to zero. Hence, all singular values are the same. From (\ref{e_srk})
\changes{in Section~\ref{s_unitnorm1}}
follows that they must be equal to one, so $\mc$  has orthonormal columns. 

%\item \emph{Stable rank of the pseudo-inverse:}
%The proof is similar to the one above.
%As for the equality, if $\mc$ has orthonormal columns, then $\mc^{\dagger}=\mc^T$,
%all $k$ singular values of $\mc$ are equal to $1$, and $\sr(\mc^{\dagger}) = k$.
%
%Conversely, if $k=\sr(\mc^{\dagger}) = 
%\|\mc^{\dagger}\|_F^2/\|\mc^{\dagger}\|_2^2$, 
%then, as above, we conclude that all singular values $1/\sigma_j(\mc)$ of $\mc^{\dagger}$
%are the same, thus all singular values 
%$\sigma_j(\mc)$ of $\mc$
%are the same, and (\ref{e_srk}) implies that 
%they are equal to one, and $\mc$ has orthonormal
%columns.

\item \emph{$p$-stable rank for $p>2$}: The definition of $p$-stable rank implies $\sr_{(p)}(\mc) \le k$. 

As for the equality, if $\mc$ has orthonormal columns, then all $k$ singular values of $\mc$ are equal to $1$, and $\sr_{(p)}(\mc) = k$, see also \cite[Example 2.1]{IpsS24}.

Conversely, if $k=\sr_{(p)}(\mc) = \|\mc\|_{(p)}^p/\|\mc\|_2^p$, then the definition of $p$-stable rank 
implies 
\begin{equation*}
\sum_{j=1}^k {\sigma_j(\mc)^p}=\|\mc\|_{(p)}^p
=k\|\mc\|_2^p=k\,\sigma_1(\mc)^p.
\end{equation*}
Hence $\sum_{j=1}^k{(\sigma_1(\mc)^p-\sigma_j(\mc)^p)}=0$.
From the singular value ordering~(\ref{e_so}) follows that this
is a sum of non-negative summands, and every
summand must be equal to zero. Hence, all singular values 
of $\mc$ are the same. From (\ref{e_srk}) \changes{in Section~\ref{s_unitnorm1}}
follows that they must be equal to one, so $\mc$  has orthonormal columns.
%The singular value ordering $\sigma_1(\mc) \ge \cdots \ge \sigma_k(\mc) = 1$ and (\ref{e_srk}) imply $\sigma_j(\mc)= 1$,  $1\leq j\leq k$. Thus $\mc$ has orthonormal columns.
%\item \emph{$p$-stable rank of the pseudo-inverse:}
%The proof is similar to the one above.
%
%As for the equality, if $\mc$ has orthonormal columns, then $\mc^{\dagger}=\mc^T$,
%all $k$ singular values equal to $1$, and $\sr_{(p)}(\mc^{\dagger}) = k$.
%
%Conversely, if $k=\sr_{(p)}(\mc^{\dagger}) = 
%\|\mc^{\dagger}\|_{(p)}^p/\|\mc^{\dagger}\|_2^p$, 
%then, as above, we conclude that all singular values $1/\sigma_j(\mc)$ of $\mc^{\dagger}$
%are the same, thus all singular values 
%$\sigma_j(\mc)$ of $\mc$
%are the same, and (\ref{e_srk}) implies that 
%they are equal to one, and $\mc$ has orthonormal
%columns.
\end{enumerate}
\end{proof}

\section{Expressions for partitioned pseudo-inverses}\label{s_partinv}
For partitioned full column-rank matrices, we derive expressions for the 
partitioned pseudo-inverses (Lemma~\ref{l_pi0}), and lower bounds for the Frobenius norm of the 
partitioned pseudo-inverse
(Lemma~\ref{l_fi}) 
and for Schatten $p$-norms for $p>2$ (Lemma~\ref{l_pi1}).

Different versions of the expressions below have appeared in \cite[Theorem 1]{baksalary2021formulae},
\cite[Theorem 2]{cline1964representations} and~\cite[Lemma 3.3 and 3.4]{cegielski2001obtuse}. 
Here we present a short, self-contained proof.

\begin{lemma}\label{l_pi0}
Let $\mc=\begin{bmatrix}\mc_1&\mc_2\end{bmatrix}\in\rmn$ 
with $\mc_1\in\real^{m\times k}$ and $\mc_2\in\real^{m\times (n-k)}$ for some $1\leq k<n$.
If $\rank(\mc)=n$, then 
\begin{equation*}
\mc^{\dagger}=\begin{bmatrix}\mm_1^{\dagger} \\ \mm_2^{\dagger}\end{bmatrix}\in\real^{n\times m},\qquad
\text{where}\quad \mm_1\equiv \mP_2\mc_1, \quad \mm_2\equiv \mP_1\mc_2
\end{equation*}
and
\begin{equation*}
\mP_1\equiv \mi-\mc_1\mc_1^{\dagger},\qquad
\mP_2\equiv \mi-\mc_2\mc_2^{\dagger}
\end{equation*}
are $m\times m$ orthogonal projectors
onto $\range(\mc_1)^{\perp}$ and $\range(\mc_2)^{\perp}$, respectively.
\end{lemma}

\begin{proof}
Since $\mc$ has full column rank, $\mc^{\dagger}=(\mc^T\mc)^{-1}\mc^T$,
and
\begin{equation*}
\mc^T\mc=\begin{bmatrix}\mc_1^T\mc_1 &\mc_1^T\mc_2\\ \mc_2^T\mc_1 & \mc_2^T\mc_2\end{bmatrix}.
\end{equation*}
Furthermore, the column submatrices $\mc_1$ and $\mc_2$ also have full column rank, so that 
$\mc_1^T\mc_1\in\real^{k\times k}$ and 
$\mc_2^T\mc_2\in\real^{(n-k)\times (n-k)}$ are 
nonsingular.

From \cite[(4)]{Cot74} follows
\begin{equation*}
(\mc^T\mc)^{-1}=\begin{bmatrix}\ms_1^{-1} & \\ &\ms_2^{-1}\end{bmatrix}
\begin{bmatrix}\mi & -\mc_1^T\mc_2(\mc_2^T\mc_2)^{-1}\\ 
-\mc_2^T\mc_1(\mc_1^T\mc_1)^{-1}&\mi\end{bmatrix},
\end{equation*}
where
\begin{align}\begin{split}\label{e_pi0a}
\ms_1&\equiv \mc_1^T\mc_1-\mc_1^T\mc_2\mc_2^{\dagger}\mc_1
=\mc_1^T\>(\mi-\mc_2\mc_2^{\dagger})\>\mc_1=\mc_1^T\mP_2\mc_1=\mm_1^T\mm_1\\
\ms_2&\equiv \mc_2^T\mc_2-\mc_2^T\mc_1\mc_1^{\dagger}\mc_2=\mm_2^T\mm_2.
\end{split}
\end{align}
The nonsingularity of the Schur complements $\ms_1$ and $\ms_2$ follows from the nonsingularity of $\mc^T\mc$, $\mc_1^T\mc_1$ and $\mc_2^T\mc_2$, see (\ref{e_sc}) and
\cite[Section 1]{Cot74}.

Multiplying $(\mc^T\mc)^{-1}$ by $\mc^T$ on the right gives
\begin{equation}\label{e_pi0}
\mc^{\dagger}=
\begin{bmatrix}\ms_1^{-1}(\mc_1^T-\mc_1^T\mc_2\mc_2^{\dagger})\\
\ms_2^{-1}(-\mc_2^T\mc_1\mc_1^{\dagger}+\mc_2^T)\end{bmatrix}=
\begin{bmatrix}(\mm_1^T\mm_1)^{-1}\mc_1^T\mP_2\\ (\mm_2^T\mm_2)^{-1}\mc_2^T\mP_1\end{bmatrix}=
\begin{bmatrix}\mm_1^{\dagger}\\ \mm_2^{\dagger}\end{bmatrix}.
\end{equation}$\quad$
\end{proof}

Below is a lower bound on the Frobenius norm in terms of the partitioned inverses.

\begin{lemma}\label{l_fi}
Let $\mc=\begin{bmatrix}\mc_1 & \mc_2\end{bmatrix}\in\rmn$
with $\mc_1\in\real^{m\times k}$ and $\mc_2\in\real^{m\times (n-k)}$ for some $1\leq k<n$.
If $\rank(\mc)=n$, then 
\begin{equation*}
\|\mc^{\dagger}\|_F^2\geq \|\mc_1^{\dagger}\|_F^2+\|\mc_2^{\dagger}\|_F^2.
\end{equation*}
\end{lemma}

\begin{proof}
Lemma~\ref{l_pi0} and (\ref{e_pi0}) imply
\begin{equation*}
\mc^{\dagger}=\begin{bmatrix}\mm_1^{\dagger} \\\mm_2^{\dagger}\end{bmatrix}, 
\qquad \text{where}\quad
\mm_1\equiv \mP_2\mc_2, \quad
\mm_2\equiv \mP_1\mc_2
\end{equation*}
and $\mP_1$ and $\mP_2$ are orthogonal projectors. 
Since $\|\mP_1\|_2=\|\mP_2\|_2=1$, the product
inequalities for singular values 
\cite[Problem 7.3.P16]{HoJoI} imply
\begin{equation*}
\sigma_i(\mm_1)=\sigma_i(\mP_2\mc_1)\leq
\|\mP_2\|_2\,\sigma_i(\mc_1)=\sigma_i(\mc_1),
\qquad 1\leq i\leq k.
\end{equation*}
Because both $\mm_1$ and $\mc_1$ have full rank, we can invert the singular values
\begin{equation}\label{eqn:eigc1m1}
1/\sigma_i(\mc_1)\leq 1/\sigma_i(\mm_1), \qquad 1\leq i\leq k.
\end{equation}
Therefore
$\|\mm_1^{\dagger}\|_F\geq\|\mc_1^{\dagger}\|_F$. This is a special case of \cite[Theorem 3.2]{Maher2007}.
The proof for $\|\mm_2^{\dagger}\|_F\geq\|\mc_2^{\dagger}\|_F$ is analogous. 
Thus
\begin{equation*}
\|\mc^{\dagger}\|_F^2
=\|\mm_1^{\dagger}\|_F^2+\|\mm_2^{\dagger}\|_F^2\geq
\|\mc_1^{\dagger}\|_F^2+\|\mc_2^{\dagger}\|_F^2.
\end{equation*}$\quad$
\end{proof}

Below is a lower bound on the Schatten $p$-norms in terms of the partitioned inverses. 

\begin{lemma}\label{l_pi1}
    Let $\mc=\begin{bmatrix}\mc_1 & \mc_2\end{bmatrix}\in\rmn$ have $\rank(\mc)=n$. Then for $p> 2$
\begin{equation*}
\|\mc^{\dagger}\|_{(p)}^2 \geq \|\mc_1^{\dagger}\|_{(p)}^2+\|\mc_2^{\dagger}\|_{(p)}^2.
\end{equation*}
\end{lemma}

\begin{proof}

%\ii{Two symmetric matrices $\ma_1, \ma_2\in\rnn$ satisfy
%the 
%Loewner partial order $\ma_1\preceq \ma_2$, if
%$\ma_2-\ma_1$ is symmetric positive semidefinite
%\cite[Section 3.1]{HoJoII}.} \aks{I assumed we would review this material earlier.}
%Applying this to the matrices $\mm_1$ and $\mm_2$ in Lemma~\ref{l_pi0}
%gives
%\[ \mm_1^T \mm_1 = \mc_1^T \mP_2 \mc_1 \preceq \mc_1^T \mc_1 \]
%because $\mc_1^T\mc_1-\mm_1^T\mm_1=\mc_1^T(\mi-\mP_2)\mc_1$
%is symmetric positive definite
%{

%Since $\mP_2$ and $\mc_1\mc_1^{\dagger}$ are orthogonal projectors, they satisfy
%$\|\mm_1\mc_1^{\dagger}\|_2=\|\mP_2\mc_1\mc_1^{\dagger}\|_2\leq \|\mP_2\|_2\|\mc_1\mc_1^{\dagger}\|_2\leq 1$, and
%}
%Since $\mc_1$ has full column rank, $ (\mm_1 \mc_1^\dagger )^T \mm_1 \mc_1^\dagger  \preceq \mi.$ Thus, $ \| \mm_1 \mc_1^\dagger\|_2 \le 1$, and 
%\[ \mm_1 \mc_1^\dagger (\mm_1 \mc_1^\dagger)^T \preceq \mi.\]
%Since $\mm_1$ has full column rank,~\cite[Theorem 7.7.2(a)]{HoJoI} implies %\ii{a congruence transformation with $\mm_1^{\dagger}$ gives}
%$\mc_1^\dagger (\mc_1^\dagger)^T \preceq \mm_1^\dagger  (\mm_1^\dagger)^T $, with analogous relation between $\mc_2$ and $\mm_2$. 
From~\eqref{eqn:eigc1m1}, follows $\sigma_j(\mc_1^\dagger)^2 \le \sigma_j(\mm_1^\dagger)^2$ for $1 \le j \le k$, or $\|\mc_1^\dagger\|_{(p)}^2 \le \|\mm_1^\dagger\|_{(p)}^2$, with an analogous relation between $\mc_2$ and $\mm_2$. Lemma~\ref{l_pi0} implies 
\[ \mc^\dagger (\mc^\dagger)^T = \begin{bmatrix} \mm_1^\dagger (\mm_1^\dagger)^T & \mm_1^{\dagger}(\mm_2^{\dagger})^T \\ \mm_2^{\dagger}(\mm_1^{\dagger})^T  & \mm_2^\dagger (\mm_2^\dagger)^T \end{bmatrix}.\]
Applying the pinching inequality~\cite[Equation (IV.52)]{Bhatia97} gives
\[ \begin{aligned} \|\mc^\dagger\|_{(p)}^2 = \|\mc^\dagger (\mc^\dagger)^T\|_{(p/2)} 
&{\ge  \|\diag\begin{pmatrix}\mm_1^\dagger (\mm_1^\dagger)^T & \mm_2^\dagger (\mm_2^\dagger)^T\end{pmatrix}\|_{(p/2)}}\\
&=\>  \|\mm_1^\dagger (\mm_1^\dagger)^T\|_{(p/2)} + \|\mm_2^\dagger (\mm_2^\dagger)^T\|_{(p/2)}\\ &\ge  
\> \|\mc_1^\dagger (\mc_1^\dagger)^T\|_{(p/2)} + \|\mc_2^\dagger (\mc_2^\dagger)^T\|_{(p/2)} \\
&= \> \|\mc_1^{\dagger}\|_{(p)}^{2} + \|\mc_2^{\dagger}\|_{(p)}^2.
\end{aligned}
\] 
%From Weyl's inequality, follows, $\|\mc_1^\dagger\|_F^2 \le \|\mm_1^\dagger\|_F^2$.
\end{proof}

\subsection*{Acknowledgments}
We are most grateful to three anonymous reviewers for their
constructive feedback in regard to exposition and literature references.

%\bibliography{ssbib}

\begin{thebibliography}{10}

\bibitem{AD2025}
{\sc R.~Armstrong and A.~Damle}, {\em Collect, commit, expand: Efficient
  {CPQR}-based column selection for extremely wide matrices}, 2025.
\newblock arXiv:2501.18035.

\bibitem{AvronBout2013}
{\sc H.~Avron and C.~Boutsidis}, {\em Faster subset selection for matrices and
  applications}, SIAM J. Matrix Anal. Appl., 34 (2013), pp.~1464--1499.

\bibitem{baksalary2021formulae}
{\sc O.~M. Baksalary and G.~Trenkler}, {\em On formulae for the
  {M}oore-{P}enrose inverse of a columnwise partitioned matrix}, Appl. Math.
  Comput., 403 (2021), pp.~Paper No. 125913, 10.

\bibitem{Bhatia97}
{\sc R.~Bhatia}, {\em Matrix analysis}, vol.~169 of Graduate Texts in
  Mathematics, Springer-Verlag, New York, 1997.

\bibitem{BDMI2014}
{\sc C.~Boutsidis, P.~Drineas, and M.~{Magdon-Ismail}}, {\em Near-optimal
  column-based matrix reconstruction}, SIAM J. Comput., 43 (2014),
  pp.~687--717.

\bibitem{civril2014column}
{\sc A.~\c{C}ivril}, {\em Column subset selection problem is {UG}-hard}, J.
  Comput. and System Sci., 80 (2014), pp.~849--859.

\bibitem{cegielski2001obtuse}
{\sc A.~Cegielski}, {\em Obtuse cones and {G}ram matrices with non-negative
  inverse}, Linear Algebra Appl., 335 (2001), pp.~167--181.

\bibitem{CI1994}
{\sc S.~Chandrasekaran and I.~C.~F. Ipsen}, {\em On rank-revealing
  factorisations}, SIAM J. Matrix Anal. Appl., 15 (1994), pp.~592--622.

\bibitem{CM09}
{\sc A.~{\c C}ivril and M.~Magdon-Ismail}, {\em On selecting a maximum volume
  sub-matrix of a matrix and related problems}, Theoret. Comput. Sci., 410
  (2009), pp.~4801--4811.

\bibitem{CMI2013}
\leavevmode\vrule height 2pt depth -1.6pt width 23pt, {\em Exponential
  inapproximability of selecting a maximum volume sub-matrix}, Algorithmica, 65
  (2013), pp.~159--176.

\bibitem{cline1964representations}
{\sc R.~E. Cline}, {\em Representations for the generalized inverse of a
  partitioned matrix}, J. Soc. Indust. Appl. Math., 12 (1964), pp.~588--600.

\bibitem{Cot74}
{\sc R.~W. Cottle}, {\em Manifestations of the {S}chur complement}, Linear
  Algebra Appl., 8 (1974), pp.~189--211.

\bibitem{DGTY2024}
{\sc A.~Damle, S.~Glas, A.~Townsend, and A.~Yu}, {\em How to reveal the rank of
  a matrix?}, 2024.
\newblock arXiv:2405.04330.

\bibitem{dHM2011}
{\sc F.~R. {de Hoog} and R.~M.~M. Mattheij}, {\em A note on subset selection
  for matrices}, Linear Algebra Appl., 434 (2011), pp.~1845--1850.

\bibitem{Dem87}
{\sc J.~W. Demmel}, {\em The geometry of ill-conditioning}, J. Complexity, 3
  (1987), pp.~201--229.

\bibitem{DeshR10}
{\sc A.~Deshpande and L.~Rademacher}, {\em Efficient volume sampling for
  row/column subset selection}, in 2010 {IEEE} 51st {A}nnual {S}ymposium on
  {F}oundations of {C}omputer {S}cience---{FOCS} 2010, IEEE Computer Soc., Los
  Alamitos, CA, 2010, pp.~329--338.

\bibitem{DeshR06}
{\sc A.~Deshpande, L.~Rademacher, S.~Vempala, and G.~Wang}, {\em Matrix
  approximation and projective clustering via volume sampling}, Theory Comput.,
  2 (2006), pp.~225--247.

\bibitem{DEFM15}
{\sc M.~Di~Summa, F.~Eisenbrand, Y.~Faenza, and C.~Moldenhauer}, {\em On
  largest volume simplices and sub-determinants}, in Proceedings of the
  {T}wenty-{S}ixth {A}nnual {ACM}-{SIAM} {S}ymposium on {D}iscrete
  {A}lgorithms, SIAM, Philadelphia, PA, 2015, pp.~315--323.

\bibitem{eswar2024bayesian}
{\sc S.~Eswar, V.~Rao, and A.~K. Saibaba}, {\em Bayesian {D}-optimal
  experimental designs via column subset selection}, arXiv preprint
  arXiv:2402.16000,  (2024).

\bibitem{Garey}
{\sc M.~R. Garey and D.~S. Johnson}, {\em Computers and intractability}, A
  Series of Books in the Mathematical Sciences, W. H. Freeman and Co., San
  Francisco, CA, 1979.
\newblock A guide to the theory of NP-completeness.

\bibitem{GovL13}
{\sc G.~H. Golub and C.~F. {Van Loan}}, {\em Matrix computations}, Johns
  Hopkins Studies in the Mathematical Sciences, Johns Hopkins University Press,
  Baltimore, MD, fourth~ed., 2013.

\bibitem{GZ2025}
{\sc L.~Grigori and Z.~Xue}, {\em Randomized strong rank-revealing {QR} for
  column subset selection and low-rank matrix approximation}, 2025.
\newblock arXiv:2503.18496.

\bibitem{GuEis96}
{\sc M.~Gu and S.~C. Eisenstat}, {\em Efficient algorithms for computing a
  strong rank-revealing {QR} factorization}, SIAM J. Sci. Comput., 17 (1996),
  pp.~848--869.

\bibitem{HLim13}
{\sc C.~J. Hillar and L.-H. Lim}, {\em Most tensor problems are {NP}-hard}, J.
  ACM, 60 (2013), pp.~Art. 45, 39.

\bibitem{HongPan92}
{\sc Y.~P. Hong and C.-T. Pan}, {\em A lower bound for the smallest singular
  value}, Linear Algebra Appl., 172 (1992), pp.~27--32.

\bibitem{HoJoI}
{\sc R.~A. Horn and C.~R. Johnson}, {\em Matrix analysis}, Cambridge University
  Press, Cambridge, second~ed., 2013.

\bibitem{IpsS24}
{\sc I.~C.~F. Ipsen and A.~K. Saibaba}, {\em Stable rank and intrinsic
  dimension of real and complex matrices}, SIAM J. Matrix Anal. Appl., 46
  (2025), pp.~1988--2007.

\bibitem{Karp1972}
{\sc R.~M. Karp}, {\em Reducibility among combinatorial problems}, in
  Complexity of computer computations ({P}roc. {S}ympos., {IBM} {T}homas {J}.
  {W}atson {R}es. {C}enter, {Y}orktown {H}eights, {N}.{Y}., 1972), The IBM
  Research Symposia Series, Plenum, New York-London, 1972, pp.~85--103.

\bibitem{KO2025}
{\sc I.~Kozyrev and A.~Osinsky}, {\em Subset selection for matrices in spectral
  norm}, 2025.
\newblock arXiv:2507.20435.

\bibitem{KressNi2024}
{\sc D.~Kressner, T.~Ni, and A.~Uschmajew}, {\em On the approximation of
  vector-valued functions by volume sampling}, J. Complexity, 86 (2025),
  p.~101887.

\bibitem{LCS2024}
{\sc J.~T. Lauzon, S.~W. Cheung, Y.~Shin, Y.~Choi, D.~M. Copeland, and
  K.~Huynh}, {\em S-{OPT}: a points selection algorithm for hyper-reduction in
  reduced order models}, SIAM J. Sci. Comput., 46 (2024), pp.~B474--B501.

\bibitem{Maher2007}
{\sc P.~J. Maher}, {\em Some singular values, and unitarily invariant norm
  inequalities concerning generalized inverses}, Filomat, 21 (2007),
  pp.~99--111.

\bibitem{NikoS16}
{\sc A.~Nikolov and M.~Singh}, {\em Maximizing determinants under partition
  constraints}, in S{TOC}'16---{P}roceedings of the 48th {A}nnual {ACM}
  {SIGACT} {S}ymposium on {T}heory of {C}omputing, ACM, New York, 2016,
  pp.~192--201.

\bibitem{NikoST19}
{\sc A.~Nikolov, M.~Singh, and U.~Tantipongpipat}, {\em Proportional volume
  sampling and approximation algorithms for a-optimal design}, in Proceedings
  of the Thirtieth Annual ACM-SIAM Symposium on Discrete Algorithms, SODA '19,
  USA, 2019, Society for Industrial and Applied Mathematics, pp.~1369--1386.

\bibitem{NikoST22}
\leavevmode\vrule height 2pt depth -1.6pt width 23pt, {\em Proportional volume
  sampling and approximation algorithms for {$A$}-optimal design}, Math. Oper.
  Res., 47 (2022), pp.~847--877.

\bibitem{Osinsky2023b}
{\sc A.~Osinsky}, {\em Polynomial time {$\rho $}-locally maximum volume
  search}, Calcolo, 60 (2023), pp.~Paper No. 42, 26.

\bibitem{Osinsky2024}
{\sc A.~Osinsky}, {\em Volume-based subset selection}, Numer. Linear Algebra
  Appl., 31 (2024), pp.~Paper No. e2525, 14.

\bibitem{Papa1991}
{\sc C.~H. Papadimitriou and M.~Yannakakis}, {\em Optimization, approximation,
  and complexity classes}, J. Comput. System Sci., 43 (1991), pp.~425--440.

\bibitem{SX2016}
{\sc Y.~Shin and D.~Xiu}, {\em Nonadaptive quasi-optimal points selection for
  least squares linear regression}, SIAM J. Sci. Comput., 38 (2016),
  pp.~A385--A411.

\bibitem{shitov2021column}
{\sc Y.~Shitov}, {\em Column subset selection is {NP-complete}}, Linear Algebra
  Appl., 610 (2021), pp.~52--58.

\bibitem{SoodHastie2025}
{\sc A.~Sood and T.~Hastie}, {\em A statistical view of column subset
  selection}, J. R. Stat. Soc. Ser. B. Stat. Methodol., 87 (2025),
  pp.~1382--1403.

\bibitem{Welch1982}
{\sc W.~J. Welch}, {\em Algorithmic complexity: three {NP}-hard problems in
  computational statistics}, J. Statist. Comput. Simulation, 15 (1982),
  pp.~17--25.

\bibitem{WillBook}
{\sc D.~P. Williamson and D.~B. Shmoys}, {\em The design of approximation
  algorithms}, Cambridge University Press, Cambridge, 2011.

\end{thebibliography}
%\bibliographystyle{siam}

\end{document}